\documentclass[12pt]{amsart}
\usepackage{}

\usepackage{amsmath}
\usepackage{amsfonts}
\usepackage{amssymb}
\usepackage[all]{xy}           

\usepackage{bbding}
\usepackage{txfonts}
\usepackage{amscd}

\usepackage[shortlabels]{enumitem}
\usepackage{ifpdf}
\ifpdf
  \usepackage[colorlinks,final,backref=page,hyperindex]{hyperref}
\else
  \usepackage[colorlinks,final,backref=page,hyperindex,hypertex]{hyperref}
\fi
\usepackage{tikz}
\usepackage[active]{srcltx}

\makeatletter

\newtheorem{thm}{Theorem}[section]
\newtheorem{lem}[thm]{Lemma}
\newtheorem{cor}[thm]{Corollary}
\newtheorem{pro}[thm]{Proposition}
\newtheorem{ex}[thm]{Example}
\newtheorem{rmk}[thm]{Remark}
\newtheorem{defi}[thm]{Definition}

\newcommand {\emptycomment}[1]{}

\newcommand{\dr}{\delta^{\rm reg}}

\newcommand{\lon }{\,\rightarrow\,}
\newcommand{\be }{\begin{equation}}
\newcommand{\ee }{\end{equation}}

\newcommand{\g}{\mathfrak g}
\newcommand{\h}{\mathfrak h}

\newcommand{\huaL}{\mathcal{L}}
\newcommand{\huaR}{\mathcal{R}}

\newcommand{\huaG}{\mathcal{G}}

\newcommand{\huaC}{{\mathcal{C}}}

\newcommand{\huaO}{{\mathcal{O}}}

\newcommand{\frkd}{\mathfrak d}

\newcommand{\frks}{\mathfrak s}

\newcommand{\frkL}{\mathfrak L}

\newcommand{\frkR}{\mathfrak R}

\newcommand{\half}{\frac{1}{2}}

\newcommand{\Id}{\rm{Id}}

\newcommand{\br}[1]{   [ \cdot,    \cdot  ]   }

\newcommand{\dt}{\delta^{\mathrm{T}}}

\newcommand{\Hom}{\mathrm{Hom}}

\newcommand{\Sym}{\mathrm{Sym}}

\newcommand{\perm}{\mathbb S}

\newcommand{\gl}{\mathfrak {gl}}

\newcommand{\kup}{$\huaO$-operator }
\newcommand{\kups}{$\huaO$-operators }

\newcommand{\ad}{\mathrm{ad}}

\newcommand{\sgn}{\mathrm{sgn}}

\newcommand{\K}{\mathbb{K}}

\newcommand{\GRB}{\huaO}

\newcommand{\MN}{\mathrm{MN}}

\begin{document}

\title[Twisting on pre-Lie algebras and quasi-pre-Lie bialgebras]{Twisting on pre-Lie algebras and quasi-pre-Lie bialgebras}

\author{Jiefeng Liu}
\address{School of Mathematics and Statistics, Northeast Normal University, Changchun 130024, China}
\email{liujf12@126.com}
\vspace{-5mm}


\begin{abstract}
We study (quasi-)twilled pre-Lie algebras and the associated $L_\infty$-algebras and differential graded Lie algebras. Then we show that certain twisting transformations on (quasi-)twilled pre-Lie algbras can be characterized by the solutions of Maurer-Cartan equations of the associated differential graded Lie algebras ($L_\infty$-algebras). Furthermore, we show that $\huaO$-operators and twisted $\huaO$-operators are solutions of the Maurer-Cartan equations. As applications, we study (quasi-)pre-Lie bialgebras using the associated differential graded Lie algebras ($L_\infty$-algebras) and the twisting theory of (quasi-)twilled pre-Lie algebras. In particular, we give a construction of quasi-pre-Lie bialgebras using symplectic Lie algebras, which is parallel to that a Cartan $3$-form on a semi-simple Lie algebra gives a quasi-Lie bialgebra.

\end{abstract}


\maketitle

\tableofcontents

\allowdisplaybreaks


\section{Introduction}\label{sec:intr}
Quasi-Lie bialgebras were introduced by Drinfeld in \cite{D90} as the classical limits of quasi-Hopf algebras.  Then Kosmann-Schwarzbach in \cite{Kosmann92} found an elegant way to express the structure of quasi-Lie bialgebras by the ``big bracket", which was first defined  by Kostant and Sternberg \cite{KS87} and introduced into the theory of Lie bialgebras by Lecomte and Roger \cite{LR90}. In 1999, Roytenberg generalized the above idea to Lie bialgebroids and quasi-Lie bialgebroids using the language of supermanifold \cite{Roy99}. The notion of twisting operation as a useful tool to construct new quasi-Hopf algebras was introduced by Drinfeld and its classical limit leads to the notion of twisting operation on quasi-Lie bialgebras, which provides a way to analyze double theory of (quasi-)Lie algebras and construct new (quasi-)Lie bialgebras \cite{BaKos,Kosmann92}. From the viewpoint of Poisson geometry, twisting operations were studied in detail by Kosmann-Schwarzbach \cite{KS05} and Roytenberg \cite{Roy02A}. See \cite{Uchino} for more details of twisting operations on associative algebras.

The aim of this paper is to develop the theory of (quasi-)pre-Lie bialgebras along the philosophy of Kosmann-Schwarzbach \cite{Kosmann92} and Roytenberg \cite{Roy02A} and to use twisting operations on pre-Lie algebras to study the double theory of (quasi-)pre-Lie algebras.

Pre-Lie algebras (or left-symmetric algebras) are a class of nonassociative algebras coming
from the study of convex homogeneous cones, affine manifolds and affine structures on
Lie groups, deformation of associative algebras and then  appeared in many fields in
mathematics and mathematical physics, such as complex and symplectic structures on Lie
groups and Lie algebras, integrable systems, Poisson brackets and infinite dimensional
Lie algebras, vertex algebras, quantum field theory, operads and so on. See \cite{Bakalov,ChaLiv,Lichnerowicz,MT}, and the survey article \cite{Pre-lie algebra in geometry} and the references therein for more details. In recent papers \cite{Ban,DSV}, the new integration theory for pre-Lie algebra was introduced as a tool to develop the deformation theory controlled by pre-Lie algebras, which can be used to prove that Kapranov¡¯s $L_\infty$ algebra structure on the Dolbeault complex of a K\"{a}hler manifold is homotopy abelian and independent of
the choice of K\"{a}hler metric up to an $L_\infty$ isomorphism.

The notion of pre-Lie bialgebra (or left-symmetric bialgebra) was introduced by Bai in \cite{Left-symmetric bialgebras}. This structure is equivalent to that of para-K\"{a}hler Lie algebra, which is a symplectic Lie algebra with a decomposition into a direct sum of the underlying vector spaces of two Lagrangian subalgebras. See \cite{BeM,Kan,Ku94,Ku99a} for more details about para-K\"{a}hler Lie algebras and applications in
mathematical physics. The notion of coboundary pre-Lie bialgebra leads to an analogue of the classical Yang-Baxter equation, called an ``S-equation''.  A  symmetric  solution of  this  equation, called an ``$\frks$-matrix'',  gives a pre-Lie bialgebra and a para-K\"{a}hler Lie algebra naturally. Recently, pre-Lie bialgebroid (or left-symmetric bialgebroid) was introduced as a geometric generalization of a pre-Lie bialgebra and the authors developed the Manin triple theory for pre-Lie algebroids \cite{LBS2}. As was pointed out in \cite{Left-symmetric bialgebras} by Bai, pre-Lie bialgebras have many properties similar to those of Lie bialgebras given by Drinfeld \cite{Dr83}. It is well known that a symplectic structure on a Lie algebra gives a non-degenerate skew-symmetric $r$-matrix. Similarly, a Hessian structure on a pre-Lie algebra, which corresponds to an affine Lie group $G$ with a $G$-invariant Hessian metric \cite{Shima}, also gives a non-degenerate symmetric $\frks$-matrix. Furthermore, since Lie bialgebra and the classical Yang-Baxter equation can be regarded as classical limits of the Hopf algebra and quantum Yang-Baxter equation \cite{Bel} respectively, the analogy mentioned above, the pre-Lie bialgebra and S-equation suggest that there might exist pre-Hopf algebra and quantum S-equation as an analogue of the Hopf algebra and quantum Yang-Baxter equation, respectively. We expect that our future study will be related to the theory of quantum groups.

We shall review the definition of (quasi-)Lie bialgebras in Kosmann-Schwarzbach's big bracket theory and twisting operations on (quasi-)Lie bialgebras. A quasi-Lie bialgebra structure on a vector space $V$ is defined by a triple $(\mu,\gamma,\varphi)$ of elements in $\wedge^\bullet(V\oplus V^*)$, with $\mu:\wedge^2 V\rightarrow V$, $\gamma:\wedge^2 V^*\rightarrow V^*$ and $\phi\in \wedge^3 V^*$, such that $\{\mu+\gamma+\phi,\mu+\gamma+\phi\}=0$, where $\{-,-\}$ is the graded Poisson bracket defined by $\{V,V\}=\{V^*,V^*\}=0$ and $\{V,V^*\}=\langle V,V^*\rangle$. \emptycomment{By the bidegree reason of $\wedge^\bullet(V\oplus V^*)$, the condition $\{\mu+\gamma+\phi,\mu+\gamma+\phi\}=0$ is equivalent to the following equations:
\begin{eqnarray*}
\half\{\mu,\mu\}+\{\gamma,\phi\}=0,\quad \{\gamma,\gamma\}=0,\quad \{\mu,\gamma\}=0,\quad\{\mu,\phi\}=0.
\end{eqnarray*}
Here the equation $\{\gamma,\gamma\}=0$ means that $\gamma$ is a Lie algebra structure on $V^*$. The equation $\half\{\mu,\mu\}+\{\gamma,\phi\}=0$ means that  $\mu$ is not a true Lie algebra structure on $V$, which only satisfies the Jacobi identity up to a defect measured by the element $\phi$. The equation $\{\mu,\gamma\}=0$ is the compatibility condition between $\mu$ and $\gamma$. Finally, $\{\mu,\phi\}=0$ is a coherent condition between $\mu$ and $\phi$.} Note that a Lie bialgebra is a quasi-Lie bialgebra with $\phi=0$. Let $\Theta=\mu+\gamma+\phi$ denote the quasi-Lie bialgebra structure on $V$. By definition, the twisting of the structure $\Theta$ by $r\in \wedge^2 V$ is a canonical transformation:
 $$\Theta^r=e^{X_{r}}(\Theta),$$
  where $X_r=\{-,r\}$ and $\Theta^r$ is the result of twisting. We recall some basic conclusions for this twisting operations. For the case $\Theta=\mu$, this structure corresponds to the Lie bialgebra such that the Lie algebra structure on $\g^*$ is trivial. If $r$ is a solution of a Maurer-Cartan equation: $$\{r,r\}_\mu=0,$$
  where $\{-,-\}_\mu:=\{\{\mu,-\},-\}$, then the result of twisting $\Theta^r=\mu+\{\mu,r\}$ is a Lie bialgebra structure on $V$. Since the bracket $\{-,-\}_\mu$ restricted to $V$ is just the Schouten bracket on the Lie algebra $(V,\mu)$, $r$ is a solution of the classical Yang-Baxter equation. Thus an $r$-matrix can be obtained through the twisting operation on a Lie bialgebra. Similarly, for the case $\Theta=\mu+\phi$, this structure corresponds to the quasi-Lie bialgebra such that the Lie algebra structure on $\g^*$ is trivial. The twisted $r$-matrix can also be obtained through the twisting operation on this quasi-Lie bialgebra \cite{Roy02A}. The twisted $r$-matrices, more generally twisted Poisson structures, appeared in string theory and topological field theory \cite{KlS,Pa}.

In Section \ref{sec:L}, using only a canonical bigraded system of the graded Poisson algebra on $\wedge^\bullet(V\oplus V^*)$, a canonical bigraded Lie algebra system on $C^*(\huaG,\huaG)=\oplus_{n\ge 1}\Hom(\wedge^{n-1}\huaG\otimes\huaG,\huaG)$ is introduced, where $\huaG=\g_1\oplus\g_2$ is a vector space decomposed into two subspaces $\g_1$ and $\g_2$ and the graded Lie bracket on $C^*(\huaG,\huaG)$ is given by Matsushima-Nijenhuis bracket. Then we introduce the notions of (quasi-)twilled pre-Lie algebras. See \cite{KosmannD} and \cite{Uchino} for more details on twilled Lie algebras and twilled associative algebras respectively. By the derived bracket construction in \cite{KosmannD}, we show that there are $L_\infty$-algebras  and differential graded Lie algebras associated to  quasi-twilled pre-Lie algebras and twilled pre-Lie algebras respectively. This is the foundation of the whole paper.

In Section \ref{sec:T}, we first study the twisting operation on a pre-Lie algebra $\huaG=\g_1\oplus\g_2$ with a decomposition  into two subspaces $\g_1$ and $\g_2$. In particular, solutions of the Maurer-Cartan equations of the aforementioned $L_\infty$-algebras and differential graded Lie algebras give rise to special twisting of quasi-twilled pre-Lie algebras and twilled pre-Lie algebras respectively. Then, we consider the semidirect product pre-Lie algebra associated to a bimodule over a pre-Lie algebra. It is a twilled pre-Lie algebra naturally and we show that the $\huaO$-operators given in \cite{Bai-Liu-Ni}, which play an important role in study of the bialgebra theory of pre-Lie algebras and $L$-dendriform algebras, are just solutions of the Maurer-Cartan equation of the associated gLa. See \cite{Bai-1,Bor,Ku99b,Uchino} for more details about $\huaO$-operators on Lie algebras and associative algebras and applications in mathematical physics. Finally, we consider the quasi-twilled pre-Lie algebra constructed by a $2$-cocycle associated to a bimodule over a pre-Lie algebra. We obtain the notion of twisted $\huaO$-operator through the solution of a Maurer-Cartan equation of the associated $L_\infty$-algebra, which is similar to the notion of twisted $r$-matrix on a Lie algebra. This part plays essential role in our later study of the theory of (quasi-)pre-Lie bialgebras.

In Section \ref{sec:P}, we first recall the notions of pre-Lie bialgebra and Manin triple for pre-Lie algebras. Note that a Manin triple for pre-Lie algebras can be seen as a twilled pre-Lie algebra with a nondegenerate invariant skew-symmetric bilinear form. For the twilled pre-Lie algebra $(\g\ltimes_{\ad^*,-R^*}\g^*,\g,\g^*)$, using the twisting theory of a twilled pre-Lie algebra given in Section \ref{sec:T} and the twisting transformation preserves the nondegenerate skew-symmetric bilinear form, we recover the notions of $S$-equation and $\frks$-matrix given in \cite{Left-symmetric bialgebras} naturally. In particular, as an $r$-matrix on a Lie algebra $\g$ can be described by the solution of a Maurer-Cartan equation of a gLa defined by the Schouten bracket on $\wedge^\bullet \g$, we find that an $\frks$-matrix on a pre-Lie algebra can also be described by the solution of a Maurer-Cartan equation of a gLa on the tensor space $\oplus_{k\ge1}(\wedge^{k-1}\g\otimes\g \otimes\g)$, which is transferred from the gLa associated to the twilled pre-Lie algebra $(\g\ltimes_{\ad^*,-R^*}\g^*,\g,\g^*)$. This is beyond the observation in \cite{Left-symmetric bialgebras}.

In Section \ref{sec:Q}, we first give the notions of quasi-pre-Lie bialgebra and quasi-Manin triple for pre-Lie algebras. Note that a quasi-Manin triple for pre-Lie algebras can be seen as a quasi-twilled pre-Lie algebra with a nondegenerate invariant skew-symmetric bilinear form. Then we show that the equivalence between quasi-pre-Lie bialgebras and quasi-Manin triples. For the quasi-twilled pre-Lie algebra $(\g\ltimes_{\ad^*,-R^*,\phi}\g^*,\g,\g^*)$ associated to a $2$-cocycle $\phi:\g\times\g \rightarrow \g^*$ with the coefficient in the bimodule $(\g^*;\ad^*,-R^*)$, using the twisting theory of a quasi-twilled pre-Lie algebra given in Section \ref{sec:T} and the twisting transformation preserves the nondegenerate skew-symmetric bilinear form, we obtain the notion of twisted $\frks$-matrix, which gives rise to a new quasi-pre-Lie bialgebra. At last, we give a construction of quasi-pre-Lie bialgebras using symplectic Lie algebras. Recall that a symplectic (Frobenius) Lie algebra  is a Lie algebra $\g$ equipped with a nondegenerate 2-cocycle $\omega\in\wedge^2\g^*$, i.e.
  $$\omega([x,y]_\g,z)+\omega([y,z]_\g,x)+\omega([z,x]_\g,y)=0,\quad\forall~x,y,z\in\g.$$ It was shown by \cite{symplectic Lie algebras} that there exists a compatible pre-Lie algebra structure $\cdot_\g$ on $\g$ given by
  \begin{equation*}
    \omega(x\cdot_\g y,z)=-\omega(y,[x,z]_\g),\quad\forall~x,y,z\in\g.
  \end{equation*}
Define $\Phi\in\wedge^2\g^*\otimes\g^*$ by
  \begin{equation*}
    \Phi(x,y,z)=\omega([x,y]_\g,z),\quad\forall~x,y,z\in\g.
  \end{equation*}
  Then $(\g,\g^*,\Phi)$  is a quasi-pre-Lie bialgebra. This result is parallel to that a Cantan $3$-form $\phi\in \wedge^3\g^*$ on a semisimple Lie algebra $\g$ defined by
  $$\phi(x,y,z)=([x,y]_\g,z),\quad\forall~x,y,z\in\g$$
  gives a quasi-Lie bialgebra, where $(-,-)$ is the Killing form on $\g$. Then we give an example of a quasi-pre-Lie bialgebra from a 4-dimensional symplectic Lie algebra and construct some twisted $\frks$-matrices on this quasi-pre-Lie bialgebra.

In this paper, all the vector spaces are over algebraically closed field $\mathbb K$ of characteristic $0$ and finite dimensional.

\vspace{2mm}

\section{Preliminaries}\label{sec:Pre}
\subsection{Cohomology of pre-Lie algebras and Matsushima-Nijenhuis brackets}
\begin{defi}  A {\bf pre-Lie algebra} is a pair $(\g,\cdot_\g)$, where $\g$ is a vector space and  $\cdot_\g:\g\otimes \g\longrightarrow \g$ is a bilinear multiplication
satisfying that for all $x,y,z\in \g$, the associator
$(x,y,z)=(x\cdot_\g y)\cdot_\g z-x\cdot_\g(y\cdot_\g z)$ is symmetric in $x,y$,
i.e.
$$(x,y,z)=(y,x,z),\;\;{\rm or}\;\;{\rm
equivalently,}\;\;(x\cdot_\g y)\cdot_\g z-x\cdot_\g(y\cdot_\g z)=(y\cdot_\g x)\cdot_\g
z-y\cdot_\g(x\cdot_\g z).$$
\end{defi}

Let $(\g,\cdot_\g)$ be a pre-Lie algebra. The commutator $
[x,y]_\g=x\cdot_\g y-y\cdot_\g x$ defines a Lie algebra structure
on $\g$, which is called the {\bf sub-adjacent Lie algebra} of
$(\g,\cdot_\g)$ and denoted by $\g^c$. Furthermore,
$L:\g\longrightarrow \gl(\g)$ with $x\mapsto L_x$, where
$L_xy=x\cdot_\g y$, for all $x,y\in \g$, gives a module of
the Lie algebra $\g^c$ of $\g$. See \cite{Pre-lie algebra in
geometry} for more details.

\begin{defi}
  Let $(\g,\cdot_\g)$ be a pre-Lie algebra and $V$   a vector space. Let $\huaL,\huaR:\g\longrightarrow\gl(V)$ be two linear maps with $x\mapsto \huaL_x$ and $x\mapsto \huaR_x$ respectively. The triple $(V;\huaL,\huaR)$ is called a {\bf bimodule} over $\g$ if
\begin{eqnarray}
\label{representation condition 1} \huaL_x\huaL_yu-\huaL_{x\cdot_\g y}u&=&\huaL_y\huaL_xu-\huaL_{y\cdot_\g x}u,\\
 \label{representation condition 2}\huaL_x\huaR_yu-\huaR_y\huaL_xu&=&\huaR_{x\cdot_\g y}u-\huaR_y\huaR_xu, \quad \forall~x,y\in \g,~ u\in V.
\end{eqnarray}
 \end{defi}

In the sequel, we will simply call $V$ a $\g$-bimodule if there is no confusions possible. It is
obvious that $( \K  ;\rho=0,\mu=0)$ is a $\g$-bimodule, which we
call the {\bf trivial bimodule}.

In fact, $(V;\huaL,\huaR)$ is a bimodule of a pre-Lie algebra
$\g$ if and only if the direct sum $\g\oplus V$ of vector spaces is
 a pre-Lie algebra (the semi-direct product) by
defining the multiplication on $\g\oplus V$ by
$$
  (x_1+v_1)\cdot_{(\huaL,\huaR)}(x_2+v_2)=x_1\cdot_\g x_2+\huaL_{x_1}v_2+\huaR_{x_2}v_1,\quad \forall~ x_1,x_2\in \g,v_1,v_2\in V.
$$
  We denote it by $\g\ltimes_{\huaL,\huaR} V$ or simply by $A\ltimes V$.

By a straightforward calculation, we have
\begin{pro}\label{pro:dual-module equiv}
 $(V;\huaL,\huaR)$ is a bimoudle over the pre-Lie algebra $(\g,\cdot_\g)$ if and only if $(V^*;\huaL^*-\huaR^*,-\huaR^*)$ is a bimoudle over the pre-Lie algebra $\g$,  where $\huaL^*:\g\longrightarrow \gl(V^*)$ and $\huaR^*:\g\longrightarrow \gl(V^*)$ are given by
$$
 \langle \huaL^*_x\xi,u\rangle=-\langle \xi,\huaL_x u\rangle,\quad\langle \huaR^*_x\xi,u\rangle=-\langle \xi,\huaR_xu\rangle,\quad \forall~ x\in \g,\xi\in V^*,u\in V.
$$
\end{pro}
In the following, $(V^*;\huaL^*-\huaR^*,-\huaR^*)$ is called the {\bf dual bimodule} of the bimodule $(V;\huaL,\huaR)$.

Let $R:\g\rightarrow
\gl(\g)$ be a linear map with $x\mapsto R_x$, where the
linear map $R_x:\g\longrightarrow\g$  is defined by
$R_x(y)=y\cdot_\g x,$ for all $x, y\in \g$. Then
$(\g;\huaL=L,\huaR=R)$ is a bimodule, which we call the
{\bf regular bimodule}. The dual bimodule of the regular bimodule $(\g;L,R)$ is just the bimodule $(\g^*;{\rm ad}^*=L^*-R^*, -R^*)$ over $\g$.

The cohomology complex for a pre-Lie algebra $(\g,\cdot_\g)$ with a bimodule $(V;\huaL,\huaR)$ is given as follows.
The set of $n$-cochains is given by
$\Hom(\wedge^{n-1}\g\otimes \g,V),\
n\geq 1.$  For all $\phi\in \Hom(\wedge^{n-1}\g\otimes \g,V)$, the coboundary operator $\delta:\Hom(\wedge^{n-1}\g\otimes \g,V)\longrightarrow \Hom(\wedge^{n}\g\otimes \g,V)$ is given by
 \begin{eqnarray}\label{eq:pre-Lie cohomology}
 \nonumber\delta\phi(x_1, \cdots,x_{n+1})&=&\sum_{i=1}^{n}(-1)^{i+1}\huaL_{x_i}\phi(x_1, \cdots,\hat{x_i},\cdots,x_{n+1})\\
\label{eq:cobold} &&+\sum_{i=1}^{n}(-1)^{i+1}\huaR_{x_{n+1}}\phi(x_1, \cdots,\hat{x_i},\cdots,x_n,x_i)\\
 \nonumber&&-\sum_{i=1}^{n}(-1)^{i+1}\phi(x_1, \cdots,\hat{x_i},\cdots,x_n,x_i\cdot_\g x_{n+1})\\
 \nonumber&&+\sum_{1\leq i<j\leq n}(-1)^{i+j}\phi([x_i,x_j]_\g,x_1,\cdots,\hat{x_i},\cdots,\hat{x_j},\cdots,x_{n+1}),
\end{eqnarray}
for all $x_i\in \g,~i=1,\cdots,n+1$. In particular, we use $\dt$ ($\dr$) to refer the coboundary operator  associated with
the trivial bimodule (the regular bimodule).

A permutation $\sigma\in\perm_n$ is called an $(i,n-i)$-unshuffle if $\sigma(1)<\cdots<\sigma(i)$ and $\sigma(i+1)<\cdots<\sigma(n)$. If $i=0$ and $i=n$, we assume $\sigma=\Id$. The set of all $(i,n-i)$-unshuffles will be denoted by $\perm_{(i,n-i)}$. The notion of an $(i_1,\cdots,i_k)$-unshuffle and the set $\perm_{(i_1,\cdots,i_k)}$ are defined analogously.

Let $\g$ be a vector space. We consider the graded vector space $C^*(\g,\g)=\oplus_{n\ge 1}C^n(\g,\g)=\oplus_{n\ge 1}\Hom(\wedge^{n-1}\g\otimes\g,\g)$. It was shown in \cite{ChaLiv,Nij,WBLS} that $C^*(\g,\g)$ equipped with the Matsushima-Nijenhuis bracket
\begin{eqnarray}\label{eq:Nij-preLie}
{[P,Q]^{\MN}}&=&P\diamond Q-(-1)^{pq}Q\diamond P,\quad\forall P\in C^{p+1}(\g,\g),Q\in C^{q+1}(\g,\g)
\end{eqnarray}
is a gLa, where $P\diamond Q\in C^{p+q+1}(\g,\g)$ is defined by
\begin{eqnarray*}
&&P\diamond Q(x_1,\ldots,x_{p+q+1})\\
&=&\sum_{\sigma\in\perm_{(q,1,p-1)}}\sgn(\sigma)P(Q(x_{\sigma(1)},\ldots,x_{\sigma(q)},x_{\sigma(q+1)}),x_{\sigma(q+2)},\ldots,x_{\sigma(p+q)},x_{p+q+1})\\
&&+(-1)^{pq}\sum_{\sigma\in\perm_{(p,q)}}\sgn(\sigma)P(x_{\sigma(1)},\ldots,x_{\sigma(p)},Q(x_{\sigma(p+1)},\ldots,x_{\sigma(p+q)},x_{p+q+1})).
\end{eqnarray*}
In particular, $\pi\in\Hom(\otimes^2\g,\g)$ defines a pre-Lie algebra if and only if $[\pi,\pi]^{\MN}=0.$ If $\pi$ is a pre-Lie algebra structure, then $d_{\pi}(f):=[\pi,f]^{\MN}$ is a graded derivation of the gLa $(C^*(\g,\g),[-,-]^{\MN})$ satisfying $d_{\pi}\circ d_{\pi}=0$, so that $(C^*(\g,\g),[-,-]^{\MN},d_{\pi})$ becomes a dgLa.

\subsection{$L_\infty$-algebras and Maurer-Cartan equations}
The notion of an $L_\infty$-algebras was introduced by Stasheff in \cite{stasheff:shla}. 

Let $V$ be a graded vector space and $TV$ be its tensor algebra over $V$. We define the graded skew-symmetric algebra over $V$ by
$$\wedge V:=TV/\langle u\otimes v-(-1)^{|u||v|} v\otimes u\rangle,$$
where $|u|$ and $|v|$ denote the degree of homogeneous elements $u,v\in V$. For any homogeneous
elements $v_1,v_2,\cdots,v_k$ and a permutation $\sigma\in S_k$, the Koszul signs $\chi(\sigma)$ is defined by
$$x_{\sigma(1)}\wedge x_{\sigma(2)}\wedge \cdots\wedge x_{\sigma(k)}=\chi(\sigma)x_1\wedge x_2\wedge \cdots\wedge x_k. $$

\begin{defi}
An {\bf $L_\infty$-algebra} is a graded vector space $\g$ equipped with a collection  of linear maps $l_k:\wedge^k\g\lon\g$ of degree $2-k$ for $k\ge 1$, such that for all $n\ge 1$
\begin{eqnarray}\label{sh-Lie}
\sum_{i+j=n+1}(-1)^{i}\sum_{\sigma\in \mathbb S_{(i,n-i)} }\chi(\sigma)l_j(l_i(x_{\sigma(1)},\cdots,x_{\sigma(i)}),x_{\sigma(i+1)},\cdots,x_{\sigma(n)})=0,
\end{eqnarray}
where $v_1,v_2,\cdots,v_n$ are homogeneous elements in $\g$.
\end{defi}
See \cite{LS,LM,Ma} for more details on $L_\infty$-algebras.

\emptycomment{
 Below are the generalized Jacobi identities in an $L_\infty$-algebra for $n=1,2,3:$
\begin{eqnarray}
 \label{weak-2}&&l_1\circ l_1=0,\\
\label{weak-3}&&l_1(l_2(x_1,x_2))-l_2(l_1(x_1),x_2)-(-1)^{|x_1|}l_2(x_1,l_1(x_2))=0,\\
\label{weak-4}&&l_2(l_2(x_1,x_2),x_3)+(-1)^{|x_1|\cdot(|x_2|+|x_3|)}l_2(l_2(x_2,x_3),x_1)+(-1)^{|x_3|\cdot(|x_1|+|x_2|)}l_2(l_2(x_3,x_1),x_2)\\
\nonumber&&-l_3(l_1(x_1),x_2,x_3)-(-1)^{|x_1|}l_3(x_1,l_1(x_2),x_3)-(-1)^{|x_1|+|x_2|}l_3(x_1,x_2,l_1(x_3))-l_1(l_3(x_1,x_2,x_3))\\
\nonumber&&=0.
\end{eqnarray}
}

\begin{defi}
Let $(\g,\{l_k\}_{k=1}^\infty)$ be an $L_\infty$-algebra and $\alpha\in \g^1$. The equation
\begin{eqnarray}\label{MC-equation}
\sum_{k=1}^{+\infty}\frac{1}{k!}l_k(\alpha,\cdots,\alpha)=0,
\end{eqnarray}
is called the {\bf Maurer-Cartan equation}.
\end{defi}

Let $(\g,[-,-]_\g,d)$ be a dgLa. We define a new degree by $|x|_d=|x|+1$ and a new bracket by
\begin{eqnarray}
[x,y]_{d}:=(-1)^{|x|_d}[dx,y]_\g,\quad \forall x,y\in\g.
\end{eqnarray}
The new bracket is called the  {\bf derived bracket} \cite{KosmannD}. 

\begin{pro}{\rm (\cite{KosmannD})}\label{old-derived}
Let $(\g,[-,-]_\g,d)$ be a dgLa, and let $\h\subset \g$ be an abelian subalgebra, i.e. $[\h,\h]_\g=0$. If the derived bracket is closed on $\h$, then $(\h,[-,-]_d)$ is a gLa.
\end{pro}

Let $V$ be a graded vector space, we define the {\bf suspension operator} $s:V\lon sV$ by assigning  $V$ to the graded vector space $sV$ with $(sV)^i:=V^{i-1}$. Similarly, the {\bf desuspension operator} $s:V\lon s^{-1}V$ is defined by assigning  $V$ to the graded vector space $s^{-1}V$ with $(s^{-1}V)^i:=V^{i+1}$.

Let $(\g,[\cdot,\cdot]_\g,d)$ be a dgLa. We set $d_t=\sum_{i=0}^{+\infty}d_it^i$ a deformation of  $d$ with $d_0=d$.
Here $d_t$ is a differential on $\g[[t]]$, which is a  graded Lie algebra of formal series with coefficients in $\g$. The condition $d_t\circ d_t=0$ is equivalent to
$$
\sum_{i+j=n\atop i,j\ge0}d_i\circ d_j=0,\quad n\in\mathbb Z_{\ge0}.
$$
 Define an $i$ ($i\ge1$)-ary derived bracket $l_i$ on $s\g$ by
\begin{eqnarray}\label{derived-bracket-1}
l_i=(-1)^{\frac{(i-1)(i-2)}{2}}s\circ N_i\circ(\overbrace{s^{-1}\otimes\cdots\otimes s^{-1}}^{i})\circ
(sd_{i-1}s^{-1}\otimes\overbrace{{\Id}\otimes\cdots\otimes{\Id}}^{i-1}),
\end{eqnarray}
where $N_i$ is defined by
\begin{eqnarray}
N_i(x_1,\cdots,x_i)=[\cdots[[x_1,x_2]_\g,x_3]_\g,\cdots,x_i]_\g,\quad \forall x_1, \cdots, x_i\in\g.
\end{eqnarray}

\begin{pro}{\rm (\cite{Uchino-1})}\label{higher-derived-bracket}
Let $(\g,[\cdot,\cdot]_\g,d)$ be a dgLa and  $\h\subset \g$  an abelian subalgebra. If $l_i$ is closed on $s\h$, then $(s\h,\{l_i\}_{i=1}^\infty)$ is an $L_\infty$-algebra.
\end{pro}

\section{Twilled and quasi-twilled pre-Lie algebras}\label{sec:L}

\subsection{Lift and bidegree}\hspace{2mm}

Let $\g_1$ and $\g_2$ be two vector spaces. The elements in $\g_1$ are denoted by $x,y, x_i$ and the elements in $\g_2$ are denoted by $u,v,v_i$. Let $c:\wedge^{n-1}\g_2\otimes\g_2\lon \g_1$ be a linear map. We define a linear map $\hat{c}\in C^n(\g_1\oplus\g_2,\g_1\oplus\g_2)$ by
\begin{eqnarray*}
\hat{c}\big((x_1,v_1),\cdots,(x_n,v_n)\big):=(c(v_1,\cdots,v_n),0).
\end{eqnarray*}
In general, for a given linear map $f:\wedge^{k-1}\g_1\otimes\wedge^{l}\g_2\otimes \g_1\lon\g_1$, we define a linear map $\hat{f}\in C^{k+l}(\g_1\oplus\g_2,\g_1\oplus\g_2)$ by
\begin{equation*}
 \hat{f}((x_1,v_1),(x_2,v_2),\cdots,(x_k,v_k))=(\sum_{\sigma\in\perm_{(k-1,l)}}\sgn(\tau)f(x_{\tau(1)},\cdots,x_{\tau(k-1)},v_{\tau(k)},\cdots,v_{\tau(k+l-1)},x_{k+l}),0).
\end{equation*}
Similarly, for $f:\wedge^{k-1}\g_1\otimes\wedge^{l}\g_2\otimes \g_1\lon\g_2$, we define a linear map $\hat{f}\in C^{k+l}(\g_1\oplus\g_2,\g_1\oplus\g_2)$ by
\begin{equation*}
 \hat{f}((x_1,v_1),(x_2,v_2),\cdots,(x_k,v_k))=(0,\sum_{\sigma\in\perm_{(k-1,l)}}\sgn(\tau)f(x_{\tau(1)},\cdots,x_{\tau(k-1)},v_{\tau(k)},\cdots,v_{\tau(k+l-1)},x_{k+l})).
\end{equation*}
For $f:\wedge^{k}\g_1\otimes\wedge^{l-1}\g_2\otimes \g_2\lon\g_1$, we define a linear map $\hat{f}\in C^{k+l}(\g_1\oplus\g_2,\g_1\oplus\g_2)$ by
\begin{equation*}
\hat{f}((x_1,v_1),(x_2,v_2),\cdots,(x_k,v_k))=(\sum_{\sigma\in\perm_{(k,l-1)}}\sgn(\tau)f(x_{\tau(1)},\cdots,x_{\tau(k)},v_{\tau(k+1)},\cdots,v_{\tau(k+l-1)},v_{k+l}),0).
\end{equation*}
Similarly, for $f:\wedge^{k}\g_1\otimes\wedge^{l-1}\g_2\otimes \g_2\lon\g_2$, we define a linear map $\hat{f}\in C^{k+l}(\g_1\oplus\g_2,\g_1\oplus\g_2)$ by
\begin{equation*}
\hat{f}((x_1,v_1),(x_2,v_2),\cdots,(x_k,v_k))=(0,\sum_{\sigma\in\perm_{(k,l-1)}}\sgn(\tau)f(x_{\tau(1)},\cdots,x_{\tau(k)},v_{\tau(k+1)},\cdots,v_{\tau(k+l-1)},v_{k+l})).
\end{equation*}
The linear map $\hat{f}$ is called a {\bf lift} of $f$. For example, the lifts of linear maps $\alpha:\g_1\otimes\g_1\lon\g_1,~\beta:\g_1\otimes\g_2\lon\g_2$ and $\gamma:\g_2\otimes\g_1\lon\g_2$ are given by
\begin{eqnarray}
\label{semi-direct-1}\hat{\alpha}\big((x_1,v_1)\otimes(x_2,v_2)\big)&=&(\alpha(x_1,x_2),0),\\
\label{semi-direct-2}\hat{\beta}\big((x_1,v_1)\otimes(x_2,v_2)\big)&=&(0,\beta(x_1,v_2)),\\
\label{semi-direct-3}\hat{\gamma}\big((x_1,v_1)\otimes(x_2,v_2)\big)&=&(0,\gamma(v_1,x_2)),
\end{eqnarray}
respectively. Let $H:\g_2\lon\g_1$ be a linear map whose lift is given by
$
\hat{H}(x,v)=(H(v),0).
$
Obviously we have $\hat{H}\circ\hat{H}=0.$\vspace{2mm}

We define $\g^{l,k}=\wedge^{l-1}\g_1\otimes\wedge^{k}\g_2\otimes \g_1+\wedge^{l}\g_1\otimes\wedge^{k-1}\g_2\otimes \g_2$. For instance,
$$
\g^{2,1}=(\wedge^2\g_1\otimes\g_2)\oplus(\g_1\otimes\g_2\otimes\g_1).
$$
 The vector space $\wedge^{n-1}{(\g_1\oplus\g_2)}\otimes (\g_1\oplus\g_2)$ is isomorphic to the direct sum of
$\g^{l,k}$, $l + k = n$. For example,
$$
\wedge^{2}(\g_1\oplus\g_2)\otimes(\g_1\oplus\g_2)=\g^{3,0}\oplus\g^{2,1}\oplus\g^{1,2}\oplus\g^{0,3}.
$$
It is obvious that
\begin{eqnarray}\label{decomposition}
C^n(\g_1\oplus\g_2,\g_1\oplus\g_2)\cong\sum_{l+k=n}C^n(\g^{l,k},\g_1)\oplus\sum_{l+k=n}C^n(\g^{l,k},\g_2).
\end{eqnarray}

\begin{defi}
A linear map $f\in \Hom\big(\wedge^{n-1}(\g_1\oplus\g_2)\otimes(\g_1\oplus\g_2),(\g_1\oplus\g_2)\big)$ has a {\bf bidegree} $l|k$, if the following four conditions hold:
\begin{itemize}
\item[\rm(i)] $l+k+1=n;$
\item[\rm(ii)] If $X$ is an element in $\g^{l+1,k}$, then $f(X)\in\g_1;$
\item[\rm(iii)] If $X$ is an element in $\g^{l,k+1}$, then $f(X)\in\g_2;$
\item[\rm(iv)] All the other case, $f(X)=0.$
\end{itemize}
We denote a linear map $f$ with bidegree $l|k$ by $||f||=l|k$.
\end{defi}
We call a linear map $f$ {\bf homogeneous} if $f$ has a bidegree.
We have $l+k\ge0,~k,l\ge-1$ because $n\ge1$ and $l+1,~k+1\ge0$. For example, the lift $\hat{H}\in C^1(\g_1\oplus\g_2,\g_1\oplus\g_2)$ of $H:\g_2\lon\g_1$ has the bidegree $-1|1$. The linear maps $\hat{\alpha},~\hat{\beta},~\hat{\gamma}\in C^2(\g_1\oplus\g_2,\g_1\oplus\g_2)$ given by \eqref{semi-direct-1}, \eqref{semi-direct-2} and \eqref{semi-direct-3} have the bidegree  $||\hat{\alpha}||=||\hat{\beta}||=||\hat{\gamma}||=1|0$. Naturally we obtain a homogeneous linear map of the bidegree $1|0$,
\begin{eqnarray}
\label{semi-direct}\hat{\mu}:=\hat{\alpha}+\hat{\beta}+\hat{\gamma}.
\end{eqnarray}
Observe that $\hat{\mu}$ is a multiplication of the semi-direct product type,
$$
\hat{\mu}\big((x_1,v_1),(x_2,v_2)\big)=(\alpha(x_1,x_2),\beta(x_1,v_2)+\gamma(v_1,x_2)).
$$
Even though $\hat{\mu}$ is not a lift (there is no $\mu$), we still use the symbol for our convenience below. \vspace{2mm}

By the definition of bidegree, we obtain the following lemmas.
\emptycomment{
\begin{lem}
  The bidegree of $f\in C^n(\g_1\oplus\g_2,\g_1\oplus\g_2)$ is $l|k$ if and only if the following four conditions hold:
\begin{itemize}
\item[\rm(i)] $l+k+1=n;$
\item[\rm(ii)] If $X$ is an element in $\g^{l+1,k}$, then $f(X)\in\g_1;$
\item[\rm(iii)] If $X$ is an element in $\g^{l,k+1}$, then $f(X)\in\g_2;$
\item[\rm(iv)] All the other case, $f(X)=0.$
\end{itemize}
\end{lem}
}

\begin{lem}\label{Zero-condition-1}
Let $f_1,\cdots,f_k\in C^n(\g_1\oplus\g_2,\g_1\oplus\g_2)$ be homogeneous linear maps and the bidegrees of $f_i$ be
different. Then $f_1+\cdots+f_k=0$ if and only if $f_1=\cdots=f_k=0.$
\end{lem}

\begin{lem}\label{Zero-condition-2}
If $||f||=-1|l$ (resp. $l|-1$) and $||g||=-1|k$ (resp. $k|-1$), then $[f,g]^{\MN}=0.$
\end{lem}

\begin{lem}\label{important-lemma-2}
If $||f||=l_f|k_f$ and $||g||=l_g|k_g$, then $[f,g]^{\MN}$ has the bidegree $l_f+l_g|k_f+k_g.$
\end{lem}

\subsection{Twilled and quasi-twilled pre-Lie algebras}
Let $(\huaG,\ast)$ be a pre-Lie algebra with a decomposition  into two subspaces $\huaG=\g_1\oplus\g_2$. Then we have
\begin{lem}\label{lem:dec}
Any $2$-cochain $\pi\in C^2(\huaG,\huaG)$ can be uniquely decomposed into four homogeneous linear maps of bidegrees $2|-1,~1|0,~0|1$ and $-1|2,$
$$
\pi=\hat{\phi}_1+\hat{\mu}_1+\hat{\mu}_2+\hat{\phi}_2.
$$
\end{lem}
\begin{proof}
By \eqref{decomposition},  the space $C^2(\huaG,\huaG)$ is decomposed into four subspaces
$$
C^2(\huaG,\huaG)=(2|-1)+(1|0)+(0|1)+(-1|2),
$$
where $(i|j)$ is the space of linear maps  of the bidegree $i|j$. Thus $\pi$ is uniquely decomposed into homogeneous linear maps of bidegrees $2|-1,~1|0,~0|1$ and $-1|2$.
\end{proof}

The pre-Lie algebra multiplication $\ast$ of $\huaG$ can be uniquely decomposed by the canonical projections $\huaG\lon\g_1$ and $\huaG\lon\g_2$ into eight multiplications:
\begin{eqnarray*}
\label{1}x\ast y&=&(x\ast_1y,x\ast_2y),\quad
\label{2}x\ast v=(x\ast_2 v,x\ast_1 v),\\
\label{3}u\ast y&=&(u\ast_2 y,u\ast_1 y),\quad
\label{4}u\ast v=(u\ast_1 v,u\ast_2 v).
\end{eqnarray*}
For convenience, we also use $\pi$ to denote the multiplication $\ast$, i.e. $$\pi((x,u),(y,v)):=(x,u)\ast(y,v).$$ Denote $\pi=\hat{\phi}_1+\hat{\mu}_1+\hat{\mu}_2+\hat{\phi}_2$ as in Lemma \ref{lem:dec}.  Then we have
\begin{eqnarray}
\label{bracket-1}\hat{\phi}_1((x,u),(y,v))&=&(0,x\ast_2y),\\
\label{bracket-2}\hat{\mu}_1((x,u),(y,v))&=&(x\ast_1 y,x\ast_1 v+u\ast_1 y),\\
\label{bracket-3}\hat{\mu}_2((x,u),(y,v))&=&(x\ast_2 v+u\ast_2 y,u\ast_2 v),\\
\label{bracket-4}\hat{\phi}_2((x,u),(y,v))&=&(u\ast_1 v,0).
\end{eqnarray}
Note that $\hat{\phi}_1$ and $\hat{\phi}_2$ are lifted linear maps of $\phi_1(x,y):=x\ast_2 y$ and $\phi_2(u,v):=u\ast_1 v$, respectively.

\begin{lem}\label{proto-twilled}
The Maurer-Cartan equation $[\pi,\pi]^{\MN}=0$ is equivalent to the following  conditions:
\begin{eqnarray}\label{eq:OT}
\left\{\begin{array}{rcl}
{}[\hat{\mu}_1,\hat{\phi}_1]^{\MN}&=&0,\\
{}\frac{1}{2}[\hat{\mu}_1,\hat{\mu}_1]^{\MN}+[\hat{\mu}_2,\hat{\phi}_1]^{\MN}&=&0,\\
{}[\hat{\mu}_1,\hat{\mu}_2]^{\MN}+[\hat{\phi}_1,\hat{\phi}_2]^{\MN}&=&0,\\
{}\frac{1}{2}[\hat{\mu}_2,\hat{\mu}_2]^{\MN}+[\hat{\mu}_1,\hat{\phi}_2]^{\MN}&=&0,\\
{}[\hat{\mu}_2,\hat{\phi}_2]^{\MN}&=&0.
\end{array}\right.
\end{eqnarray}
\end{lem}
\begin{proof}
By Lemma \ref{Zero-condition-2}, we have
\begin{eqnarray*}
[\pi,\pi]^{\MN}&=&[\hat{\phi}_1+\hat{\mu}_1+\hat{\mu}_2+\hat{\phi}_2,\hat{\phi}_1+\hat{\mu}_1+\hat{\mu}_2+\hat{\phi}_2]^{\MN}\\
               &=&[\hat{\phi}_1,\hat{\mu}_1]^{\MN}+[\hat{\phi}_1,\hat{\mu}_2]^{\MN}+[\hat{\phi}_1,\hat{\phi}_2]^{\MN}+[\hat{\mu}_1,\hat{\phi}_1]^{\MN}+[\hat{\mu}_1,\hat{\mu}_1]^{\MN}\\
               &&+[\hat{\mu}_1,\hat{\mu}_2]^{\MN}+[\hat{\mu}_1,\hat{\phi}_2]^{\MN}+[\hat{\mu}_2,\hat{\phi}_1]^{\MN}+[\hat{\mu}_2,\hat{\mu}_1]^{\MN}+[\hat{\mu}_2,\hat{\mu}_2]^{\MN}\\
               &&+[\hat{\mu}_2,\hat{\phi}_2]^{\MN}+[\hat{\phi}_2,\hat{\phi}_1]^{\MN}+[\hat{\phi}_2,\hat{\mu}_1]^{\MN}+[\hat{\phi}_2,\hat{\mu}_2]^{\MN}\\
               &=&(2[\hat{\mu}_1,\hat{\phi}_1]^{\MN})+([\hat{\mu}_1,\hat{\mu}_1]^{\MN}+2[\hat{\mu}_2,\hat{\phi}_1]^{\MN})+(2[\hat{\mu}_1,\hat{\mu}_2]^{\MN}+2[\hat{\phi}_1,\hat{\phi}_2]^{\MN})\\
               &&+([\hat{\mu}_2,\hat{\mu}_2]^{\MN}+2[\hat{\mu}_1,\hat{\phi}_2]^{\MN})+(2[\hat{\mu}_2,\hat{\phi}_2]^{\MN}).
\end{eqnarray*}
By Lemma \ref{important-lemma-2} and Lemma \ref{Zero-condition-1},  $[\pi,\pi]^{\MN}=0$ if and only if  \eqref{eq:OT} holds.
\end{proof}

\begin{defi}
Let $(\huaG,\pi)$ be a pre-Lie algebra with a decomposition  into two subspaces $\huaG=\g_1\oplus\g_2$ and the multiplication structure
$\pi=\hat{\phi}_1+\hat{\mu}_1+\hat{\mu}_2+\hat{\phi}_2$.
The triple $(\huaG,\g_1,\g_2)$ is called a {\bf quasi-twilled pre-Lie algebra} if $\phi_2=0$, or equivalently, $\g_2$ is a subalgebra.
 \end{defi}

By Lemma \ref{proto-twilled}, we have
\begin{lem}\label{lem:quasi-t}
The triple $(\huaG,\g_1,\g_2)$ is a quasi-twilled pre-Lie algebra  if and only if the following four conditions hold:
\begin{eqnarray}
\label{quasi-1}[\hat{\mu}_1,\hat{\phi}_1]^{\MN}&=&0,\\
\label{quasi-2}\frac{1}{2}[\hat{\mu}_1,\hat{\mu}_1]^{\MN}+[\hat{\mu}_2,\hat{\phi}_1]^{\MN}&=&0,\\
\label{quasi-3}[\hat{\mu}_1,\hat{\mu}_2]^{\MN}&=&0,\\
\label{quasi-4}\frac{1}{2}[\hat{\mu}_2,\hat{\mu}_2]^{\MN}&=&0.
\end{eqnarray}
\end{lem}

The following theorem shows that a quasi-twilled pre-Lie algebra gives an $L_\infty$-algebra.

\begin{thm}\label{quasi-as-shLie}
Let $(\huaG,\g_1,\g_2)$ be a quasi-twilled pre-Lie algebra. We define $d_{\hat{\mu}_2}:C^m(\g_2,\g_1)\lon C^{m+1}(\g_2,\g_1)$, $[-,-]_{{\hat{\mu}_1}}:C^m(\g_2,\g_1)\times C^n(\g_2,\g_1)\lon C^{m+n}(\g_2,\g_1)$ and $[-,-,-]_{\hat{\phi}_1}:C^m(\g_2,\g_1)\times C^n(\g_2,\g_1)\times C^k(\g_2,\g_1)\lon C^{m+n+k-1}(\g_2,\g_1)$ by
\begin{eqnarray}
\label{eq:shLie1}d_{\hat{\mu}_2}(f_1)&=&[\hat{\mu}_2,\hat{f}_1]^{\MN},\\
\label{eq:shLie2}{[f_1,f_2]_{\hat{\mu}_1}}&=&(-1)^{m-1}[[\hat{\mu}_1,\hat{f}_1]^{\MN},\hat{f}_2]^{\MN},\\
\label{eq:shLie3}{[f_1,f_2,f_3]}_{\hat{\phi}_1}&=&(-1)^{n-1}[[[\hat{\phi}_1,\hat{f}_1]^{\MN},\hat{f}_2]^{\MN},\hat{f}_3]^{\MN},
\end{eqnarray}
for all $f_1\in C^m(\g_2,\g_1),~f_2\in C^n(\g_2,\g_1),~f_3\in C^k(\g_2,\g_1).$ Then $(C^*(\g_2,\g_1),d_{\hat{\mu}_2},[-,-]_{{\hat{\mu}_1}},[-,-,-]_{\hat{\phi}_1})$ is an $L_\infty$-algebra.
\end{thm}
\begin{proof}
We set $d_0:=[\hat{\mu}_2,-]^{\MN}$. By \eqref{quasi-4} and the fact that $(C^*(\huaG,\huaG),[-,-]^{\MN})$ is a gLa, we deduce that $(C^*(\huaG,\huaG),[-,-]^{\MN},d_0)$ is a dgLa. Moreover, we define
$$
d_1:=[\hat{\mu}_1,-]^{\MN},\quad d_2:=[\hat{\phi}_1,-]^{\MN},\quad d_i=0,\quad\forall i\ge3.
$$
It is straightforward to check that
$
\sum_{i+j=n\atop i,j\ge0}d_i\circ d_j=0,\,\,n\in\mathbb Z_{\ge0}.
$
Therefore, we have higher derived brackets on $sC^*(\huaG,\huaG)$ given by
\begin{eqnarray*}
d_{\hat{\mu}_2}(sf_1)&=&s[\hat{\mu}_2,f_1]^{\MN},\\
{[sf_1,sf_2]_{\hat{\mu}_1}}&=&(-1)^{|f_1|}s[[\hat{\mu}_1,f_1]^{\MN},f_2]^{\MN},\\
{[sf_1,sf_2,sf_3]}_{\hat{\phi}_1}&=&(-1)^{|f_2|}s[[[\hat{\phi}_1,f_1]^{\MN},f_2]^{\MN},f_3]^{\MN},\\
l_i&=&0,\,\,\,\,i\ge4,
\end{eqnarray*}
where $f_1,~f_2,~f_3\in C^*(\huaG,\huaG)$. By Lemma \ref{important-lemma-2}, $d_{\hat{\mu}_2},[-,-]_{{\hat{\mu}_1}},[-,-,-]_{\hat{\phi}_1}$ are closed on $C^*(\g_2,\g_1)$. Moreover, $C^*(\g_2,\g_1)$ is an abelian subalgebra of the gLa $(C^*(\huaG,\huaG),[-,-]^{\MN})$. By Proposition \ref{higher-derived-bracket}, our claim follows.
\end{proof}

\begin{defi}
Let $(\huaG,\pi)$ be a pre-Lie algebra with a decomposition  into two subspaces $\huaG=\g_1\oplus\g_2$ and the multiplication structure
$\pi=\hat{\phi}_1+\hat{\mu}_1+\hat{\mu}_2+\hat{\phi}_2$. The triple $(\huaG,\g_1,\g_2)$ is called a {\bf twilled pre-Lie algebra} if $\phi_1=\phi_2=0$, or equivalently, $\g_1$ and $\g_2$ are subalgebras of $\huaG$.
\end{defi}

By  Lemma \ref{proto-twilled}, we have
\begin{lem}\label{lem:twillL}
The triple $(\huaG,\g_1,\g_2)$ is a twilled pre-Lie algebra if and only if the following three conditions hold:
\begin{eqnarray}
\label{twilled-1}\frac{1}{2}[\hat{\mu}_1,\hat{\mu}_1]^{\MN}&=&0,\\
\label{twilled-2}[\hat{\mu}_1,\hat{\mu}_2]^{\MN}&=&0,\\
\label{twilled-3}\frac{1}{2}[\hat{\mu}_2,\hat{\mu}_2]^{\MN}&=&0.
\end{eqnarray}
\end{lem}

\begin{rmk}\label{two-representation}
By \eqref{twilled-1}, we obtain that $\hat{\mu}_1$ is a  pre-Lie algebra multiplication on $\huaG=\g_1\oplus\g_2$. Then, by \eqref{bracket-2}, $\hat{\mu}_1|_{\g_1\otimes\g_1}$ is a  pre-Lie algebra multiplication on $\g_1$, which we denote by $(\g_1,\ast_1)$ and $\huaL_xv:=\hat{\mu}_1(x,v),~\huaR_xv:=\hat{\mu}_1(v,x)$ is a bimodule of $(\g_1,\ast_1)$ on the vector space $\g_2$. Similarly, $\hat{\mu}_2|_{\g_2\otimes\g_2}$ is a  pre-Lie algebra multiplication on $\g_2$, which we denote by $(\g_2,\ast_2)$ and $\frkL_vx:=\hat{\mu}_2(v,x),~\frkR_vx:=\hat{\mu}_2(x,v)$ is a bimodule of $(\g_2,\ast_2)$ on the vector space $\g_1$.

\end{rmk}

\begin{cor}\label{twilled-DGLA}
Let  $(\huaG,\g_1,\g_2)$ be a twilled pre-Lie algebra. Then $(C^*(\g_2,\g_1),d_{\hat{\mu}_2},[-,-]_{\hat{\mu}_1})$ is  a dgLa, where $d_{\hat{\mu}_2}$ and $[-,-]_{\hat{\mu}_1}$ are given by \eqref{eq:shLie1} and \eqref{eq:shLie2} respectively.
\end{cor}

Let $(V;\huaL,\huaR)$ be a bimodule of a pre-Lie algebra $\g$. It is obvious that the semi-direct pre-Lie algebra $\g\ltimes_{\huaL,\huaR} V$ is a twilled pre-Lie algebra. Let $\mu$ denote the semi-direct product structure on $\g\oplus V$. Thus
\begin{cor}\label{semi-direct-GLA}
  Let $\g\ltimes_{\huaL,\huaR} V$ be a semi-direct pre-Lie algebra associated to the bimodule $(V;\huaL,\huaR)$ over the pre-Lie algebra $\g$. Then
  $(C^*(V,\g),[-,-]_{\hat{\mu}})$ is a gLa, where $[-,-]_{\hat{\mu}}$ is given by \eqref{eq:shLie2}.
\end{cor}

The following lemma gives the precise formulas of the dgLa $(C^*(\g_2,\g_1),d_{\hat{\mu}_2},[-,-]_{\hat{\mu}_1})$ given by Corollary \ref{twilled-DGLA}, which is useful in  characterization of \kups as the solutions of Maurer-Cartan equations
and in the definition of $S$-equations and $\frks$-matrices on a pre-Lie algebra.
\begin{lem}\label{twilled-DGLA-concrete}
 Let  $(\huaG,\g_1,\g_2)$ be a twilled pre-Lie algebra with the structure $\pi=\hat{\mu}_1+\hat{\mu}_2$. Then $(C^*(\g_2,\g_1),d_{\hat{\mu}_2},[-,-]_{\hat{\mu}_1})$ is a dgLa, where $d_{\hat{\mu}_2}$ and $[-,-]_{\hat{\mu}_1}$ are given by
\begin{eqnarray}
\label{eq:twilled-DGLA-concrete1} d_{\hat{\mu}_2}(f_1) ( v_1,\cdots, v_{m+1} )&=&\sum_{i=1}^{m}(-1)^{m+i}\frkL_{v_i}f_1(v_1, \cdots,\hat{v_i},\cdots,v_{m+1})\\
\nonumber&&+\sum_{i=1}^{m}(-1)^{m+i}\huaR_{v_{m+1}}f_1(v_1, \cdots,\hat{v_i},\cdots,v_m,v_i)\\
\nonumber&&-\sum_{i=1}^{m}(-1)^{m+i}f_1(v_1, \cdots,\hat{v_i},\cdots,v_n,v_i\cdot_\g v_{m+1})\\
\nonumber &&+\sum_{1\leq i<j\leq m}(-1)^{m+i+j+1}f_1([v_i,v_j]_\g,v_1,\cdots,\hat{v_i},\cdots,\hat{v_j},\cdots,v_{m+1})
\end{eqnarray}
and
\begin{eqnarray}
\label{eq:twilled-DGLA-concrete2}&&[f_1,f_2]_{\hat{\mu}_1} (v_1,\cdots, v_{m+n} )\\
\nonumber&=&\sum_{\sigma\in\mathbb S_{(n-1,1,m-1)}}\sgn(\sigma)\Big(f_2(v_{\sigma(1)},\cdots,v_{\sigma(n)})\ast_1f_2(v_{\sigma(n+1)},\cdots,v_{\sigma(m+n-1)},v_{m+n})\\
\nonumber&&+\sum_{i=2}^m(-1)^{i-1}f_1(\huaL_{f_2(v_{\sigma(1)},\cdots,v_{\sigma(n)})}v_{\sigma(n+i-1)},v_{\sigma(n+i)},\cdots,\hat{v}_{\sigma(n+j-1)},\cdots,v_{\sigma(m+n-1)},v_{m+n})\\
\nonumber&&-f_1(v_{\sigma(n+1)},\cdots,v_{\sigma(m+n+1)},\huaL_{f_2(v_{\sigma(1)},\cdots,v_{\sigma(n)})}v_{m+n})\Big)\\
\nonumber&&+\sum_{\sigma\in\mathbb S_{(m,n-1)}}\sum_{i=1}^m\sgn(\sigma)(-1)^{m(n-1)-i-1}\Big(f_1(v_{\sigma(1)},\cdots,\hat{v}_{\sigma(i)},\cdots,v_{\sigma(m)},\huaR_{f_2(v_{\sigma(m+1)},\cdots,v_{\sigma(m+n-1)},v_{m+n})}v_{\sigma(i)})\\
\nonumber&&+f_1(v_{\sigma(1)},\cdots,\hat{v}_{\sigma(i)},\cdots,v_{\sigma(m)},v_{\sigma(i)})\ast_1f_2(v_{\sigma(m+1)},\cdots,v_{\sigma(m+n-1)},v_{m+n})\Big)\\
\nonumber&&+\sum_{\sigma\in\mathbb S_{(m,1,n-2)}}\sum_{i=1}^m\sgn(\sigma)(-1)^{m(n-1)-i}\Big(f_2(\huaR_{f_1(v_{\sigma(1)},\cdots,\hat{v}_{\sigma(i)},\cdots,v_{\sigma(m)},v_{\sigma(i)})}v_{\sigma(i)},v_{\sigma(m+1)},\cdots,v_{\sigma(m+n-1)},v_{m+n})\\
\nonumber&&+f_2(\huaL_{f_1(v_{\sigma(1)},\cdots,\hat{v}_{\sigma(i)},\cdots,v_{\sigma(m)},v_{\sigma(i)})}v_{\sigma(m+1)},v_{\sigma(m+2)},\cdots,v_{\sigma(m+n-1)},v_{m+n})\Big)\\
\nonumber&&+\sum_{\sigma\in\mathbb S_{(m,n)}}\sum_{i=1}^m\sgn(\sigma)(-1)^{i}\Big(f_2(v_{\sigma(1)},\cdots,v_{\sigma(n-1)},\huaL_{f_1(v_{\sigma(n)},\cdots,\hat{v}_{\sigma(n+i-1)},\cdots,v_{\sigma(m+n-1)},v_{\sigma(n+i-1)})}v_{m+n})\\
\nonumber&&+f_2(v_{\sigma(1)},\cdots,v_{\sigma(n-1)},\huaR_{f_1(v_{\sigma(n)},\cdots,\hat{v}_{\sigma(n+i-1)},\cdots,v_{\sigma(m+n-1)},v_{m+n})}v_{\sigma(n+i-1)})\Big)
\end{eqnarray}
for all  $f_1\in C^m(\g_2,\g_1),~f_2\in C^n(\g_2,\g_1)$.
\end{lem}
The proof follows from a complicated computation and we omit the details.

\section{Twisting on pre-Lie algebras}\label{sec:T}
Let $(\huaG,\pi)$ be a pre-Lie algebra with a decomposition  into two subspaces $\huaG=\g_1\oplus\g_2$ and the multiplication structure
$\pi=\hat{\phi}_1+\hat{\mu}_1+\hat{\mu}_2+\hat{\phi}_2$. Let $\hat{H}$ be the lift of a linear map $H:\g_2\lon\g_1$. We set
$$e^{X_{\hat{H}}}(-):={\Id}+X_{\hat{H}}+\frac{1}{2!}X^2_{\hat{H}}+\frac{1}{3!}X^3_{\hat{H}}+\cdots,$$
where $X_{\hat{H}}:=[-,\hat{H}]^{\MN}$, $X^2_{\hat{H}}:=[[-,\hat{H}]^{\MN},\hat{H}]^{\MN}$ and $X_{\hat{H}}^n$ is defined similarly. Since $\hat{H}\circ \hat{H}=0$, the operator $e^{X_{\hat{H}}}$ is well-defined.

\begin{defi}
The transformation $\pi^{H}:=e^{[-,\hat{H}]^{\MN}}\pi$ is called a {\bf twisting} of $\pi$ by $H$.
\end{defi}
It is obvious that the result of twisting of $H$ is also a $2$-cochain. The following results are followed from standard arguments in deformation theory.
\begin{lem}
$\pi^{H}=e^{-\hat{H}}\circ \pi\circ (e^{\hat{H}}\otimes e^{\hat{H}})$.
\end{lem}
\begin{proof}
For any $(x_1,v_1),~(x_2,v_2)\in\huaG$, we have
\begin{eqnarray*}
&&[\pi,\hat{H}]^{\MN}\big((x_1,v_1),(x_2,v_2)\big)=(\pi\diamond\hat{H}-\hat{H}\diamond\pi)\big((x_1,v_1),(x_2,v_2)\big)\\
                                                 &&\qquad\qquad=\pi((H(v_1),0),(x_2,v_2))+\pi((x_1,v_1),(H(v_2),0))-\hat{H}(\pi((x_1,v_1),(x_2,v_2))).
\end{eqnarray*}
By $\hat{H}\circ\hat{H}=0$, we have
\begin{eqnarray*}
[[\pi,\hat{H}]^{\MN},\hat{H}]^{\MN}\big((x_1,v_1),(x_2,v_2)\big)&=&[\pi,\hat{H}]^{\MN}((H(v_1),0),(x_2,v_2))+[\pi,\hat{H}]^{\MN}((x_1,v_1),(H(v_2),0))\\
&&-\hat{H}\big([\pi,\hat{H}]^{\MN}((x_1,v_1),(x_2,v_2))\big)\\
&=&2\pi((H(v_1),0),(H(v_2),0))-2\hat{H}\pi((H(v_1),0),(x_2,v_2))\\
&&-2\hat{H}\pi((x_1,v_1),(H(v_2),0)).
\end{eqnarray*}
Furthermore, we have
\begin{eqnarray*}
[[[\pi,\hat{H}]^{\MN},\hat{H}]^{\MN},\hat{H}]^{\MN}\big((x_1,v_1),(x_2,v_2)\big)&=&-6\hat{H}\pi((H(v_1),0),(H(v_2),0)),\\
(([\cdot,\hat{H}]^{\MN})^i\pi)\big((x_1,v_1),(x_2,v_2)\big)&=&0,\,\,\,\,\forall i\ge4,
\end{eqnarray*}
and
\begin{eqnarray}\label{twisting-operator}
e^{[\cdot,\hat{H}]^{\MN}}\pi=\pi+[\pi,\hat{H}]^{\MN}+\half[[\pi,\hat{H}]^{\MN},\hat{H}]^{\MN}+\frac{1}{6}[[[\pi,\hat{H}]^{\MN},\hat{H}]^{\MN},\hat{H}]^{\MN}.
\end{eqnarray}
Thus, we have
\begin{eqnarray*}
\pi^{H}&=&\pi-\hat{H}\circ\pi+\pi\circ(\hat{H}\otimes{\Id})+\pi\circ({\Id}\otimes\hat{H})-\hat{H}\circ\pi\circ({\Id}\otimes\hat{H})-\hat{H}\circ\pi\circ(\hat{H}\otimes{\Id})\\
          \nonumber&&+\pi\circ(\hat{H}\otimes\hat{H})-\hat{H}\circ\pi\circ(\hat{H}\otimes\hat{H}).
\end{eqnarray*}
By $\hat{H}\circ\hat{H}=0$, we have
\begin{eqnarray*}
e^{-\hat{H}}\circ \pi\circ (e^{\hat{H}}\otimes e^{\hat{H}})&=&({\Id}-\hat{H})\circ\pi\circ (({\Id}+\hat{H})\otimes({\Id}+\hat{H}))\\
&=&\pi+\pi\circ({\Id}\otimes\hat{H})+\pi\circ(\hat{H}\otimes{\Id})+\pi\circ(\hat{H}\otimes\hat{H})\\
&&-\hat{H}\circ\pi-\hat{H}\circ\pi\circ({\Id}\otimes\hat{H})-\hat{H}\circ\pi\circ(\hat{H}\otimes{\Id})-\hat{H}\circ\pi\circ(\hat{H}\otimes\hat{H}).
\end{eqnarray*}
Therefore, we obtain that $\pi^{H}=e^{-\hat{H}}\circ \pi\circ (e^{\hat{H}}\otimes e^{\hat{H}})$.
\end{proof}

\begin{pro}
The twisting $\pi^{H}$ is a  pre-Lie algebra structure on $\huaG$, i.e. $[\pi^{H},\pi^{H}]^{\MN}=0.$
\end{pro}
\begin{proof}
By $\pi^{H}=e^{-\hat{H}}\circ \pi\circ (e^{\hat{H}}\otimes e^{\hat{H}})$, we have
\begin{eqnarray*}
[\pi^{H},\pi^{H}]^{\MN}=2\pi^{H}\diamond\pi^{H}&=&2e^{-\hat{H}}\circ(\pi\diamond\pi)\circ(e^{\hat{H}}\otimes e^{\hat{H}}\otimes e^{\hat{H}})\\
&=&e^{-\hat{H}}\circ[\pi,\pi]^{\MN}\circ(e^{\hat{H}}\otimes e^{\hat{H}}\otimes e^{\hat{H}})=0.
\end{eqnarray*}
\end{proof}

Then we obtain the following useful lemma.
\begin{cor}\label{twisting-isomorphism}
$
e^{\hat{H}}:(\huaG,\pi^{H})\lon(\huaG,\pi)
$
is an isomorphism between  pre-Lie algebras.
\end{cor}


Apparently,  $\pi^{H}$ can also be decomposed into the unique four substructures. The twisting operations are completely determined by the following theorem.

\begin{thm}\label{thm:twist}
Let $\pi:=\hat{\phi}_1+\hat{\mu}_1+\hat{\mu}_2+\hat{\phi}_2$ and $\pi^{H}:=\hat{\phi}_1^{H}+\hat{\mu}_1^{H}+\hat{\mu}_2^{H}+\hat{\phi}_2^{H}$. Then we have
\begin{eqnarray}
\label{twisting-1}\hat{\phi}_1^{H}&=&\hat{\phi}_1,\\
\label{twisting-2}\hat{\mu}_1^{H}&=&\hat{\mu}_1+[\hat{\phi}_1,\hat{H}]^{\MN},\\
\label{twisting-3}\hat{\mu}_2^{H}&=&\hat{\mu}_2+d_{\hat{\mu}_1}\hat{H}+\half[[\hat{\phi}_1,\hat{H}]^{\MN},\hat{H}]^{\MN},\\
\label{twisting-4}\hat{\phi}_2^{H}&=&\hat{\phi}_2+d_{\hat{\mu}_2}\hat{H}+\half[\hat{H},\hat{H}]_{\hat{\mu}_1}+\frac{1}{6}[[[\hat{\phi}_1,\hat{H}]^{\MN},\hat{H}]^{\MN},\hat{H}]^{\MN},
\end{eqnarray}
where $d_{\hat{\mu}_i}:=[\hat{\mu}_i,-]^{\MN}~(i=1,2)$ and $[\hat{H},\hat{H}]_{\hat{\mu}_1}:=[[\hat{\mu}_1,\hat{H}]^{\MN},\hat{H}]^{\MN}$.
\end{thm}
\begin{proof}
  By the decomposition of bidegree and Lemmas \ref{Zero-condition-1}-\ref{important-lemma-2}, the theorem follows.
\end{proof}

In the following, we discuss various cases of twisting operations.
\subsection{Twisting on twilled pre-Lie algebras}
Let $\huaG=\g_1\bowtie\g_2 $ be a twilled algebra with the multiplication structure $\pi=\mu_1+\mu_2$. The twisted structures by $H:\g_2\rightarrow \g_1$ have the forms:
\begin{eqnarray}
\label{twilled-twisting-1}\hat{\mu}_1^{H}&=&\hat{\mu}_1,\\
\label{twilled-twisting-2}\hat{\mu}_2^{H}&=&\hat{\mu}_2+d_{\hat{\mu}_1}\hat{H},\\
\label{twilled-twisting-3}\hat{\phi}_2^{H}&=&d_{\hat{\mu}_2}\hat{H}+\half[\hat{H},\hat{H}]_{\hat{\mu}_1}.
\end{eqnarray}

In Corollary \ref{twilled-DGLA}, we show that there is a dgLa structure $(C^*(\g_2,\g_1),d_{\hat{\mu}_2},[-,-]_{\hat{\mu}_1})$ associated to the twilled pre-Lie algebra $(\huaG,\g_1,\g_2)$. In the following, we show that a solution of the Maurer-Cartan equation in this dgLa can give a new twilled pre-Lie algebra through the twisting operation.
\begin{pro}\label{twisting-twilled}
Let $(\huaG,\g_1,\g_2)$ be a  twilled   pre-Lie algebra   and $H:\g_2\longrightarrow\g_1$ a linear map. The   twisting $((\huaG, \pi^{H}),\g_1,\g_2)$ is a twilled pre-Lie algebra if and only if $H$ is a solution of the Maurer-Cartan equation in the dgLa $(C^*(\g_2,\g_1),d_{\hat{\mu}_2},[-,-]_{\hat{\mu}_1})$  given in Corollary \ref{twilled-DGLA}, i.e.
\begin{equation}\label{eq:Twist-Twilled-MC}
  d_{\hat{\mu}_2}\hat{H}+\half[\hat{H},\hat{H}]_{\hat{\mu}_1}=0.
\end{equation}
The condition \eqref{eq:Twist-Twilled-MC} is equivalent with
 \begin{eqnarray}\label{eq:MC-expression}
  H(u)\ast_1 H(v)+H(u)\ast_2 v+u\ast_2 H(v)=H\big(H(u)\ast_1 v+u\ast_1 H(v)\big)+H(u\ast_2 v).
  \end{eqnarray}
\end{pro}
\begin{proof}
 By a direct calculation, we have
 \begin{eqnarray*}
 d_{\hat{\mu}_2}(\hat{H})(u,v)&=&-H(u)\ast_2 v-u\ast_2 H(v)+H(u\ast_2 v);\\
 {[\hat{H},\hat{H}]_{\hat{\mu}_1}}(u,v)&=&2H\big(H(u)\ast_1 v+u\ast_1 H(v)\big)-2H(u)\ast_1 H(v).
 \end{eqnarray*}
 Thus
 \begin{eqnarray*}
  d_{\hat{\mu}_2}\hat{H}(u,v)+\half[\hat{H},\hat{H}]_{\hat{\mu}_1}(u,v)&=&-H(u)\ast_2 v-u\ast_2 H(v)+H(u\ast_2 v)\\
  &&+H\big(H(u)\ast_1 v+u\ast_1 H(v)\big)-H(u)\ast_1 H(v)=0.
 \end{eqnarray*}
\end{proof}

\begin{cor}\label{cor:pre-Lie and O-operator}
Let $(\huaG,\g_1,\g_2)$ be a  twilled   pre-Lie algebra and $H$ a solution of the Maurer-Cartan equation in the associated dgLa. Then
\begin{eqnarray}\label{eq:mul2}
u\cdot^{H}v:=u\ast_2 v+H(u)\ast_1 v+u\ast_1 H(v),\quad\forall~u,v\in\g_2
\end{eqnarray}
defines a  pre-Lie algebra structure on $\g_2$.
\end{cor}
\begin{proof}
By Lemma \ref{lem:twillL}, we deduce that $\hat{\mu}_2^{H}$ is a  pre-Lie algebra multiplication on $\huaG$. Furthermore, the multiplication restricted to $\g_2$ is given by \eqref{eq:mul2}.
 \end{proof}

At the end of this subsection, we consider the case $\hat{\mu}_2=0$. In this case, the twilled pre-Lie algebra $(\huaG,\g_1,\g_2)$ is a semi-direct pre-Lie algebra. In the following, we recall the notion of an $\huaO$-operator on a bimodule over a pre-Lie algebra given in \cite{Bai-Liu-Ni}.

\begin{defi}
    A linear map $T:V\longrightarrow \g$ is called an {\bf $\GRB$-operator} on a bimodule
    $(V;\huaL,\huaR)$ over a pre-Lie algebra $(\g,\cdot_\g)$  if it satisfies
  \begin{equation}\label{O-operator}
    T(u)\cdot_\g T(v)=T(\huaL_{T(u)}v+\huaR_{T(v)}u),\quad\forall~ u,v\in V.
  \end{equation}
\end{defi}

 Let $(V;\huaL,\huaR)$ be a bimodule over a pre-Lie algebra $(\g,\cdot_\g)$. Let $\mu$ denote the semi-direct product structure on $\g\oplus V$. Then
\begin{pro}
A linear map $T:V\longrightarrow \g$ is an $\huaO$-operator on the bimodule $(V;\huaL,\huaR)$ if and only if $\hat{T}$ is a solution of the Maurer-Cartan equation in the gLa $(C^*(V,\g),[-,-]_{\hat{\mu}})$ given in Corollary \ref{semi-direct-GLA}, i.e. $[\hat{T},\hat{T}]_{\hat{\mu}}=0$.
\end{pro}
\begin{proof}
It follows from the following relation:
$$\half[\hat{T},\hat{T}]_{\hat{\mu}}(u,v)= T(u)\cdot_\g T(v)-T(\huaL_{T(u)}v+\huaR_{T(v)}u),\quad \forall~u,v\in V.$$
\end{proof}

Furthermore, by Corollary \ref{cor:pre-Lie and O-operator}, we have
\begin{cor}{\rm(\cite{Bai-Liu-Ni})}\label{o-multiplication}
  Let $T$ be an \kup on a bimodule $(V;\huaL,\huaR)$ over a pre-Lie algebra $(\g,\cdot_\g)$. Then $(V,\cdot^T)$ is a pre-Lie algebra, where $\cdot^T$ is given by
\begin{equation}\label{eq:pre-Lie operation1}
  u\cdot^T v= \huaL_{T(u)}v+\huaR_{T(v)}u,\quad\forall~u,v\in V
\end{equation}
and $T$ is a homomorphism from the pre-Lie algebra $(V,\cdot^T)$ to the pre-Lie algebra $(\g,\cdot_\g)$.
\end{cor}

\emptycomment{By Proposition \ref{pro:construction twilled pre-Lie1}, we have
\begin{pro}\label{pro:construction twilled pre-Lie2}
  Let $\g\ltimes_{\huaL,\huaR} V$ be a semi-direct pre-Lie algebra equipped with the structure $\pi=\hat{\mu}$. Then $(\g\oplus V,\pi^H=\hat{\mu}+[\hat{\mu},\hat{H}]^{\MN})$ is a twilled pre-Lie algebra if and only if $H$ satisfies
  \begin{equation}
    d_{\hat{\mu}}([\hat{H},\hat{H}]_{\hat{\mu}})=0.
  \end{equation}
\end{pro}}

By the fact $[\hat{T},\hat{T}]_{\hat{\mu}}=0$ and Proposition \ref{twisting-twilled}, we have
\begin{cor}\label{cor:O-twilled pre-Lie}
  Let $T$ be an \kup on a bimodule $(V;\huaL,\huaR)$ over a pre-Lie algebra $(\g,\cdot_\g)$. Then $(\g\oplus V,\pi^T=\hat{\mu}+[\hat{\mu},\hat{T}]^{\MN})$ is a twilled pre-Lie algebra.
\end{cor}

Given an \kup $T:V\lon\g$, we have a twilled pre-Lie algebra $\g\bowtie V_T$ by twisting $\g\ltimes V$ by $T$, where $(V_T,\cdot^T)$ is the pre-Lie algebra given by \eqref{eq:pre-Lie operation1}.

Define $\frkL^T, \huaR^T:V\rightarrow \gl(\g)$ by
 \begin{eqnarray}
  \label{eq:LT}\frkL^T_{u}y:&=&[\hat{\mu},\hat{T}]^{\MN}((0,u),(y,0))=T(u)\cdot_{\g}y-T(\huaR_y u),\\
   \label{eq:RT}\frkR^T_{v}x:&=&[\hat{\mu},\hat{T}]^{\MN}((x,0),(0,v))=x\cdot_{\g}T(v)-T(\huaL_x v),\quad\forall ~x,y\in \g,u,v\in V.
 \end{eqnarray}
 Then we have
 \begin{pro}\label{pro:bimodule and product-O}
 $(\g; \frkL^T, \frkR^T)$ is a bimodule over the pre-Lie algebra $(V,\cdot^T)$  and the pre-Lie algebra structure on $\g\bowtie V_T$ is explicitly given by
   \begin{equation}
     (x,u)\ast(y,v)=(x\cdot_{\g}y+\frkL^T_{u}y+\frkR^T_{v}x,\huaL_x v+\huaR_yu+u\cdot^T v),\quad\forall~x,y\in\g,u,v\in V.
   \end{equation}
 \end{pro}

 \emptycomment{Let $T$ be an \kup on a bimodule $(V;\huaL,\huaR)$ over a pre-Lie algebra $(\g,\cdot_\g)$. Since $T$ is a solution of the Maurer-Cartan equation of the gLa $(\huaC^*(V,\g),[\cdot,\cdot]_{\hat{\mu}})$, it follows from the graded Jacobi identity that the map
\begin{equation}
d_{\hat{T}}:\Hom(\otimes^{n-1}V\wedge V,\g)\longrightarrow \Hom(\otimes^{n}V\wedge V,\g), \quad d_{\hat{T}}f:=[\hat{T},f]_{\hat{\mu}},
\end{equation}
 is a graded derivation of the gLa $(\huaC^*(V,\g),[\cdot,\cdot]_{\hat{\mu}})$ satisfying $d^2_{\hat{T}}=0$.
  Thus, we have
  \begin{lem}\label{lem:dgla}
 Let $T:V\longrightarrow\g$ be an \kup~ on a bimodule $(V;\huaL,\huaR)$ over a pre-Lie algebra $\g$. Then $(\huaC^*(V,\g),[\cdot,\cdot]_{\hat{\mu}},d_{\hat{T}})$ is a dgLa.
  \end{lem}

\begin{pro}\label{pro:O-SMC}
   Let $T$ be an $\huaO$-operator on a bimodule $(V;\huaL,\huaR)$ over $\g$. Then for a linear map $H:V\lon\g$,  $T+H$ is an \kup if and only if $H$ is a solution of the Maurer-Cartan equation on the dgLa $(C^*(V,\g),[-,-]_{\hat{\mu}},d_{\hat{T}})$.
\end{pro}

\begin{proof}
Note that $T+H$ is an \kup if and only if
$$[\hat{T}+\hat{H},\hat{T}+\hat{H}]_{\hat{\mu}}=0.$$
Since $[\hat{T},\hat{T}]_{\hat{\mu}}=0$, the above condition is equivalent to
$$[\hat{T},\hat{H}]_{\hat{\mu}}+\half [\hat{H},\hat{H}]_{\hat{\mu}}=0.$$
By the relation $d_T(\hat{H})=[\hat{T},\hat{H}]_{\hat{\mu}}$, the above condition is equivalent to
$$ d_T(\hat{H})+\half [\hat{H},\hat{H}]_{\hat{\mu}}=0.$$
 \end{proof}}

\subsection{Twisting on quasi-twilled pre-Lie algebras}
 Let $\huaG=\g_1\bowtie\g_2 $ be a quasi-twilled pre-Lie algebra with the structure $\pi=\mu_1+\mu_2+\phi_1$. The twisted structures by $H:\g_2\rightarrow \g_1$ have the forms:
\begin{eqnarray*}
\hat{\phi}_1^{H}&=&\hat{\phi}_1,\\
\hat{\mu}_1^{H}&=&\hat{\mu}_1+[\hat{\phi}_1,\hat{H}]^{\MN},\\
\hat{\mu}_2^{H}&=&\hat{\mu}_2+d_{\hat{\mu}_1}\hat{H}+\half[[\hat{\phi}_1,\hat{H}]^{\MN},\hat{H}]^{\MN},\\
\hat{\phi}_2^{H}&=&d_{\hat{\mu}_2}\hat{H}+\half[\hat{H},\hat{H}]_{\hat{\mu}_1}+\frac{1}{6}[[[\hat{\phi}_1,\hat{H}]^{\MN},\hat{H}]^{\MN},\hat{H}]^{\MN}.
\end{eqnarray*}

Since $\hat{\mu}_1$ is not a pre-Lie algebra structure, the derived bracket $[-,-]_{\hat{\mu}_1}$ does not satisfy the graded Jacobi rule in general. However, the space $C^*(\g_2,\g_1)$ still has an $L_\infty$-algebra structure $(d_{\hat{\mu}_2},[-,-]_{{\hat{\mu}_1}},[-,-,-]_{\hat{\phi}_1})$ given in Theorem \ref{quasi-as-shLie}. Then we have

\begin{pro}\label{pro:quasi-twilled to quasi-twilled}
The result of twisting $((\huaG,\pi^{H}),\g_1,\g_2)$ is also a quasi-twilled pre-Lie algebra   if and only if $H$ is a solution of the Maurer-Cartan equation in the above $L_\infty$-algebra, i.e.
\begin{equation}\label{eq:Twist-quasi-Twilled-MC}
d_{\hat{\mu}_2}(\hat{H})+\half[\hat{H},\hat{H}]_{{\hat{\mu}_1}}+\frac{1}{6}[\hat{H},\hat{H},\hat{H}]_{\hat{\phi}_1}=0.
\end{equation}
The condition \eqref{eq:Twist-quasi-Twilled-MC} is equivalent with
 \begin{eqnarray}\label{eq:MC-expression2}
  \nonumber&&H(u)\ast_1 H(v)+H(u)\ast_2 v+u\ast_2 H(v)\\
 &&=H\big(H(u)\ast_1 v+u\ast_1 H(v)\big)+H(u\ast_2 v)+H(\phi_1(H(u),H(v))).
  \end{eqnarray}
\end{pro}

\begin{cor}\label{cor:twist-O-operator}
Let $(\huaG,\g_1,\g_2)$ be a  quasi-twilled pre-Lie algebra and $H:\g_2\longrightarrow\g_1$ a solution of the Maurer-Cartan equation in the associated  $L_\infty$-algebra. Then
\begin{eqnarray*}
u\cdot^{{H,\phi_1}}v:=u\ast_2 v+H(u)\ast_1 v+u\ast_1 H(v)+\phi_1(H(u),H(v)),\quad\forall u,v\in\g_2
\end{eqnarray*}
defines a  pre-Lie algebra structure on $\g_2$.
\end{cor}
\begin{proof}
By  Lemma \ref{lem:quasi-t}, we deduce that $\hat{\mu}_2^{H}$ is a  pre-Lie algebra multiplication on $\huaG$. By the definition of $\hat{\mu}_2^{H}$, we obtain that $\hat{\mu}_2^{H}|_{\g_2\otimes\g_2}$ is a  pre-Lie multiplication on $\g_2$. Furthermore, the multiplication on $\g_2$ is given by:
$$
\hat{\mu}_2^{H}(u,v)=u\ast_2 v+H(u)\ast_1 v+u\ast_1 H(v)+\phi_1(H(u),H(v)).
$$
\end{proof}

In the following, we consider the case $\hat{\mu}_2=0$. By a direct calculation, we have
\begin{pro}\label{pro:condition of quasi-twilled pre-Lie}
    Let $(\g,\cdot_\g)$ be a pre-Lie algebra and $V$ a vector space. Let $\huaL,\huaR:\g\longrightarrow\gl(V)$ be two linear maps with $x\mapsto \huaL_x$ and $x\mapsto \huaR_x$ respectively and $\phi:\g\times \g\rightarrow V$ be a linear map. Then $\g\oplus V$  with the multiplication $\ast$ given by
 \begin{equation}\label{eq:extension mult}
  (x,u)\ast(y,v)=(x\cdot_\g y,\huaL_xv+\huaR_y u+\phi(x,y))
\end{equation}
is a quasi-twilled pre-Lie algebra if and only if $(V;\huaL,\huaR)$ is a bimodule over the pre-Lie algebra $\g$ and $\phi$  is a $2$-cocycle associated to the bimodule $(V;\huaL,\huaR)$. We denote this quasi-twilled pre-Lie algebra by $\g\ltimes_{\huaL,\huaR,\phi}V$.
\end{pro}
  \begin{defi}
    Let $(\g,\cdot_\g)$ be a pre-Lie algebra and $(V;\huaL,\huaR)$ a bimodule over $\g$. Let $\phi:\g\times \g\rightarrow V$ be a $2$-cocycle associated to the bimodule $(V;\huaL,\huaR)$. A linear map $T:V\rightarrow \g$ is called a {\bf $\phi$-twisted $\huaO$-operator} if it satisfies
    \begin{equation}\label{eq:twist-Operator}
    T(u)\cdot_\g T(v)=T\big(\huaL_{T(u)}v+\huaR_{T(v)}u\big)+T(\phi(T(u),T(v))),\quad\forall~u,v\in V.
  \end{equation}
  \end{defi}

Let $(V;\huaL,\huaR)$ be a bimodule over a pre-Lie algebra $(\g,\cdot_\g)$ and $\phi$ a $2$-cocycle associated to the bimodule $(V;\huaL,\huaR)$. Let $\mu+\phi$ denote the structure of the quasi-twilled pre-Lie algebra $\g\ltimes_{\huaL,\huaR,\phi}V$.
\begin{pro}
A linear map $T:V\longrightarrow \g$ is a $\phi$-twisted $\huaO$-operator on a bimodule $(V;\huaL,\huaR)$ over a pre-Lie algebra $\g$ if and only if $\hat{T}$ is a solution of the Maurer-Cartan equation on the $L_\infty$-algebra $(C^*(V,\g),d_{\hat{\mu}_2},[-,-]_{{\hat{\mu}_1}},[-,-,-]_{\hat{\phi}_1})$  given in Theorem \ref{quasi-as-shLie}, i.e. $$\half[\hat{T},\hat{T}]_{\hat{\mu}}+\frac{1}{6}[\hat{T},\hat{T},\hat{T}]_{\hat{\phi}}=0,$$
where $\mu_1=\mu$, $\mu_2=0$ and $\phi_1=\phi$.
\end{pro}
\begin{proof}
  It follows by a direct calculation, we omit the details.
\end{proof}

By Corollary \ref{cor:twist-O-operator}, we have
\begin{cor}\label{cor:twist-pre-Lie operation}
  Let $T$ be a $\phi$-twisted $\huaO$-operator on a bimodule $(V;\huaL,\huaR)$ over a pre-Lie algebra $(\g,\cdot_\g)$. Then $(V,\cdot^{T,\phi})$ is a pre-Lie algebra, where $\cdot^{T,\phi}$ is given by
\begin{equation}\label{eq:pre-Lie operation2}
  u\cdot^{T,\phi} v= \huaL_{T(u)}v+\huaR_{T(v)}u+\phi(T(u),T(v)),\quad\forall~u,v\in V,
\end{equation}
and $T$ is a homomorphism from the pre-Lie algebra $(V,\cdot^{T,\phi})$ to the pre-Lie algebra $(\g,\cdot_\g)$.
\end{cor}
By Proposition \ref{pro:quasi-twilled to quasi-twilled}, we have
\begin{cor}\label{cor:twist-o-quasi-pre-Lie}
  Let $T$ be a $\phi$-twisted $\huaO$-operator on a bimodule $(V;\huaL,\huaR)$ over a pre-Lie algebra $(\g,\cdot_\g)$. Let $\pi^\phi=\mu+\phi$ be the structure of the quasi-pre-Lie algebra $\g\ltimes_{\huaL,\huaR,\phi}V$. Then $(\g\oplus V,\pi^{T,\phi}=\hat{\pi^\phi}+[\hat{\mu},\hat{T}]^{\MN}+[\hat{\phi},\hat{T}]^{\MN})$ is a quasi-twilled pre-Lie algebra. Moreover, the quasi-twilled pre-Lie algebra structure on $\g\oplus V$ is explicitly given by
   \begin{eqnarray*}
     (x,u)\ast(y,v)&=&(x\cdot_{\g}y+\frkL^T_{u}y+\frkR^T_{v}x-T(\phi(x,y))-T(\phi(T(\xi),y))\\
     &&-T(\phi(x,T(\eta))),\huaL_x v+\huaR_yu+\phi(T(\xi),y)+\phi(x,T(\eta))+u\cdot^{T,\phi} v+\phi(x,y)),
   \end{eqnarray*}
   where $\frkL^T$ and $\frkR^T$ are given by \eqref{eq:LT} and \eqref{eq:RT} respectively and $x,y\in\g,\xi,\eta\in\g^*$.
\end{cor}
\begin{proof}
  It follows by a direct calculation. We omit the details.
\end{proof}

\begin{pro}
 Let $(\g,\cdot_\g)$ be a pre-Lie algebra and $(V;\huaL,\huaR)$ a bimodule over $\g$. Let $\varphi:\g\rightarrow V$ be an invertible $1$-cochain. Then the inverse $\varphi^{-1}$ is a $\phi$-twisted $\huaO$-operator, in this case, $\phi=-\delta \varphi$.
 \end{pro}
 \begin{proof}
   We put $T=\varphi^{-1}$. The condition \eqref{eq:twist-Operator} is equivalent to
   $$\varphi(T(u)\cdot_\g T(v))=\huaL_{T(u)}v+\huaR_{T(v)} u+\phi(T(u),T(v)),$$
   which can be written by
   $$\varphi(T(u)\cdot_\g T(v))-\huaL_{T(u)}v-\huaR_{T(v)} u=\phi(T(u),T(v)).$$
   On the other hand, we have
   $$\delta\varphi(x,y)=-\varphi(x\cdot_\g y)+\huaL_{x}\varphi(y)-\huaR_{y} \varphi(x),\quad\forall~x,y\in\g.$$
   We obtain that $$\delta\varphi(T(u),T(v))=-\phi(T(u),T(v)).$$
 \end{proof}

\section{Pre-Lie bialgebras and $\frks$-matrices}\label{sec:P}
In this section, we use the approach of twisting operations on pre-Lie algebras to study pre-Lie bialgebras and Manin triples for pre-Lie algebras.
\begin{defi}{\rm(\cite{Left-symmetric bialgebras})}\label{defi:pre-Lie bialgebra}
A {\bf pre-Lie bialgebra} on $(\g,\g^*)$ consists of the following data:
\begin{itemize}
\item[$\bullet$]A pre-Lie algebra $(\g,\cdot_\g)$ and a linear map $\beta:\g^*\rightarrow \g^*\otimes \g^*$ such that
$\langle\beta(\xi),x\otimes y\rangle=\langle\xi, x\cdot_\g y\rangle$;
\item[$\bullet$]A pre-Lie algebra $(\g^*,\cdot_{\g^*})$ and a linear map $\alpha:\g\rightarrow \g\otimes \g$ such that
$\langle\alpha(x),\xi\otimes \eta\rangle=\langle x, \xi\cdot_{\g^*} \eta\rangle$
\end{itemize}
satisfying the following properties:
\begin{itemize}
\item[$\rm(i)$]$\alpha$ is a $1$-cocycle associated to the module $(\g\otimes\g;L\otimes1+1\otimes
\ad)$ over the sub-adjacent Lie algebra $\g^c$;
\item[$\rm(ii)$]$\beta$ is a $1$-cocycle associated to the module $(\g^*\otimes\g^*;L\otimes1+1\otimes
\ad)$ over the sub-adjacent Lie algebra ${\g^*}^c$.
\end{itemize}
\end{defi}

A nondegenerate  skew-symmetric bilinear form
$(-,-)_-$ on a pre-Lie algebra $(\frkd,\cdot_\frkd)$
is called {\bf invariant} if
\begin{equation}
(a\cdot_\frkd b,c)_-+(b,[a,c]_\frkd)_-=0,\quad\forall~a,b,c\in\frkd.
\end{equation}

\begin{defi}
  A {\bf quadratic pre-Lie algebra} is a triple $(\frkd,\cdot_\frkd,(-,-)_-)$, where $(\frkd,\cdot_\frkd)$ is a pre-Lie algebra and $(-,-)_-$ is a nondegenerate invariant skew-symmetric bilinear form on $\frkd$.
\end{defi}

\begin{ex}{\rm
 Let $(\g,\cdot_\g)$ be a pre-Lie algebra. Then $(\frkd=\g\ltimes_{\ad^*,-R^*} \g^*,(-,-)_-)$ is a quadratic pre-Lie algebra, where the nondegenerate invariant skew-symmetric bilinear form $(-,-)_-$ is given by
 \begin{eqnarray}\label{symplectic bracket}
((x,\xi),(y,\eta))_-=\langle \xi,y\rangle-\langle x,\eta\rangle, \quad \forall ~x,y\in \g,~~\xi,\eta\in\g^*.
\end{eqnarray}}
\end{ex}

\begin{defi}
A   {\bf Manin triple for pre-Lie algebras}
 is a triple $((\frkd,\cdot_\frkd, (-,-)_-),\g_1,\g_2)$, where $(\frkd,\cdot_\frkd,(-,-)_-)$ is an even dimensional quadratic pre-Lie algebra, $\g_1$ and $\g_2$ are pre-Lie subalgebras, both isotropic with respect to $(-,-)_-$, and $\frkd=\g_1\oplus\g_2$ as vector spaces.
\end{defi}

\begin{rmk}
Let $((\frkd,\cdot_\frkd, (-,-)_-),\g_1,\g_2)$ be a Manin triple for pre-Lie algebras. Then $\frkd=\g_1\oplus\g_2$ is a twilled pre-Lie algebra. Thus a  Manin triple for pre-Lie algebras can be seen as a twilled pre-Lie algebra with a nondegenerate invariant skew-symmetric bilinear form.
\end{rmk}

\begin{thm}\label{thm:equivalence1}{\rm(\cite{LBS2,Left-symmetric bialgebras})}
  There is a one-to-one correspondence between  Manin triples for pre-Lie algebras and pre-Lie bialgebras.
\end{thm}

Let $T:\g^*\lon\g$ be an \kup~ on the bimodule $(\g^*;\ad^*,-R^*)$ over a pre-Lie algebra $(\g,\cdot_{\g})$. Let $\pi$ be the structure multiplication of the twilled pre-Lie algebra $\g\ltimes_{\ad^*,-R^*}\g^*$. By Corollary \ref{cor:O-twilled pre-Lie},   $(( \g\oplus\g^*,\pi^T),\g,\g^*)$  is a twilled pre-Lie algebra. Moreover, by Corollary \ref{twisting-isomorphism},
$e^{\hat{T}}:( \g\oplus\g^*,\pi^T)\lon( \g\oplus\g^*,\pi)$ is a pre-Lie algebra isomorphism.

By Corollary \ref{o-multiplication} and Proposition \ref{pro:bimodule and product-O}, we have

\begin{pro}\label{pro:algrep}
Let $T:\g^*\lon\g$ be an $\huaO$-operator on the bimodule $(\g^*;\ad^*,-R^*)$ over $\g$. Then $\g^*_T:=(\g^*,\cdot^T)$ is a pre-Lie algebra, where $\cdot^T$ is given by
$$
\xi\cdot^T \eta=\ad^*_{T(\xi)}\eta-R^*_{T(\eta)}\xi,\quad\forall~\xi,\eta\in\g^*.
$$
Furthermore, ${\ad^*}^T$ and $-{R^*}^T$ given by
\begin{eqnarray}\label{eq:dual bimodule in bialgebra}
{\ad_\xi^*}^T x= T(\xi)\cdot_\g x+T(R^*_x \xi) \quad -{R_\xi^*}^T x=x\cdot_\g T(\xi)-T(\ad^*_x \xi),\quad\forall x\in\g
\end{eqnarray}
is a bimodule of the pre-Lie algebra  $\g^*_T$ on the vector space $\g$. Moreover, the twilled pre-Lie algebra multiplication on $\g\oplus\g^*$ is explicitly given by
  \begin{equation}
     (x,\xi)\ast(y,\eta)=(x\cdot_{\g}y+{\ad_\xi^*}^T y-{R_\eta^*}^T x,\ad^*_x \eta-R^*_y\xi+\xi\cdot^{T} \eta),
   \end{equation}
   where $x,y\in\g,\xi,\eta\in \g^*$.
\end{pro}

\begin{lem}\label{lem:O-symmetric}
 Let $T:\g^*\lon\g$ be an $\huaO$-operator on the bimodule $(\g^*;\ad^*,-R^*)$ over $\g$. Then $e^{\hat{T}}$ preserves the bilinear form $(-,-)_-$ in \eqref{symplectic bracket} if and only if $T=T^*$. Here $T^*$ is the dual map of $T$, i.e.
$\langle T(\xi),\eta\rangle=\langle \xi,T^*(\eta)\rangle,$ for all $\xi,\eta\in\g^*.$
\end{lem}
\begin{proof}
  By $\hat{T}\circ\hat{T}=0$, we have $e^{\hat{T}}={\Id}+\hat{T}$. For all $x,y\in\g,~\xi,\eta\in\g^*$, we have
\begin{eqnarray*}
(e^{\hat{T}}(x+\xi),e^{\hat{T}}(y+\eta))_-&=&(x+\xi+T(\xi),y+\eta+T(\eta))_-\\
&=&(x+\xi,y+\eta)_-+(\xi,T(\eta))_-+(T(\xi),\eta)_-\\
&=&(x+\xi,y+\eta)_-+\langle \xi,T(\eta)\rangle-\langle T(\xi),\eta\rangle\\
&=&(x+\xi,y+\eta)_-+\langle (T^*-T)\xi,\eta\rangle.
\end{eqnarray*}
Thus, $(e^{\hat{T}}(x+\xi),e^{\hat{T}}(y+\eta))_-=(x+\xi,y+\eta)_-$ if and only if $T=T^*.$
\end{proof}

\begin{thm}\label{thm:isomorphism-invariant-twilled}
Let $T:\g^*\lon\g$ be an $\huaO$-operator on the bimodule $(\g^*;\ad^*,-R^*)$ over $\g$ and $T=T^*$.
Then $(\frkd=\g\oplus\g^*,\pi^T)$  is a quadratic pre-Lie algebra with the invariant bilinear form $(-,-)_-$ in \eqref{symplectic bracket} and $e^{\hat{T}}$ is an isomorphism from the quadratic pre-Lie algebra $(\frkd=\g\oplus\g^*,\pi^T)$ to $(\frkd=\g\oplus\g^*,\pi)$.
 \end{thm}

\begin{proof}
Since $e^{\hat{T}}$ is a pre-Lie algebra isomorphism and preserves the bilinear form $(-,-)_-$, for all $a,b,c\in\frkd$, we have
\begin{eqnarray*}
 (\pi^T(a,b),c)_-&=&(e^{-\hat{T}}\pi(e^{\hat{T}}a,e^{\hat{T}}b),c)_-=(\pi(e^{\hat{T}}a,e^{\hat{T}}b),e^{\hat{T}}c)_-\\
  &=&-(e^{\hat{T}}b,\pi(e^{\hat{T}}a,e^{\hat{T}}c))_-+(e^{\hat{T}}b,\pi(e^{\hat{T}}c,e^{\hat{T}}a))_-\\
  &=&-(b,e^{-\hat{T}}\pi(e^{\hat{T}}a,e^{\hat{T}}c))_-+(b,e^{-\hat{T}}\pi(e^{\hat{T}}c,e^{\hat{T}}a))_-\\
  &=&-(b,\pi^T(a,c)-\pi^T(c,a))_-,
\end{eqnarray*}
which implies that $(\frkd=\g\oplus\g^*,\pi^T)$  is a quadratic  pre-Lie algebra. It is obvious that $e^{\hat{T}}$ is an isomorphism from the quadratic pre-Lie algebra $(\frkd=\g\oplus\g^*,\pi^T)$ to $(\frkd=\g\oplus\g^*,\pi)$.
\end{proof}

By Theorem \ref{thm:equivalence1}, Proposition \ref{pro:algrep} and Theorem \ref{thm:isomorphism-invariant-twilled}, we obtain
\begin{cor}\label{cor:bialg}
  Let $T:\g^*\lon\g$ be an $\huaO$-operator on the bimodule $(\g^*;\ad^*,-R^*)$ over $\g$ and $T=T^*$. Then $(\g,\g^*_T)$ is a pre-Lie bialgebra.
\end{cor}

By Lemma \ref{twilled-DGLA-concrete}, we have
\begin{cor}\label{dual-GLA}
Let $(\g,\cdot_\g)$ be a pre-Lie algebra and denote the structure of the semi-direct product pre-Lie algebra $\g\ltimes_{\ad^*,-R^*} \g^*$ by $\mu$. Then  $(C^*(\g^*,\g),[-,-]_\mu)$ is a gLa, where $[-,-]_\mu$ is given by \eqref{eq:twilled-DGLA-concrete2} with $\huaL=\ad^*$ and $\huaR=-R^*$.
\end{cor}
In the following, we transfer the above gLa structure to the following tensor space.

For $k\ge1$, we define $\Psi:\wedge^{k-1}\g\otimes\g \otimes\g\longrightarrow \Hom(\wedge^{k-1}\g^*\otimes\g^*,\g)$ by
\begin{equation*}\label{eq:defipsi}
 \langle\Psi(P)(\xi_1,\cdots,\xi_{k-1},\xi_k),\xi_{k+1}\rangle=\langle P,~\xi_1\wedge\cdots\wedge\xi_{k-1}\otimes\xi_{k+1}\otimes\xi_{k}\rangle,\quad\forall \xi_1,\cdots, \xi_{k+1}\in\g^*,
\end{equation*}
and $\Upsilon:\Hom(\wedge^{k-1}\g^*\otimes\g^*,\g)\longrightarrow \wedge^{k-1}\g\otimes\g \otimes\g$ by
\begin{equation*}\label{eq:defiUpsilon}
 \langle\Upsilon(f),\xi_1\wedge\cdots\wedge\xi_{k-1}\otimes\xi_{k}\otimes\xi_{k+1}\rangle=\langle f(\xi_1,\cdots,\xi_{k+1}),\xi_{k}\rangle,\quad \forall \xi_1,\cdots, \xi_{k+1}\in\g^*.
\end{equation*}
Obviously we have $\Psi\circ\Upsilon={\Id},~~\Upsilon\circ\Psi={\Id}.$

\begin{thm}
Let $(\g,\cdot_\g)$ be a pre-Lie algebra. Then, there is a graded Lie  bracket $\llbracket -,-\rrbracket$ on the tensor space $\oplus_{k\ge1}(\wedge^{k-1}\g\otimes\g \otimes\g)$ given by
$$
\llbracket P,Q\rrbracket:=\Upsilon[\Psi(P),\Psi(Q)]_{\mu},\,\,\,\,\forall P\in\wedge^{m-1}\g\otimes\g \otimes\g,Q\in\wedge^{n-1}\g\otimes\g \otimes\g.
$$
\end{thm}
The general formula of $\llbracket P,Q\rrbracket$ is rather complicated. But for $P=x\otimes y$ and $Q=z\otimes w$, there is a concrete expression.
\begin{lem}\label{gla-s-matrix}
For $x\otimes y,~z\otimes w\in\g\otimes\g$, we have
\begin{eqnarray}\label{2-tensor}
\nonumber\llbracket x\otimes y,z\otimes w\rrbracket&=&z\otimes y\otimes [w,x]_\g+x\otimes w\otimes [y,z]_\g+z\otimes w\cdot_\g y\otimes x\\
&&+x\otimes y\cdot_\g w\otimes z -x\cdot_\g w\otimes y\otimes z-z\cdot_\g y\otimes w\otimes x.
\end{eqnarray}
\end{lem}

Moreover, we can obtain the tensor form of an $\huaO$-operator on a bimodule $(\g^*;\ad^*,-R^*)$ over $\g$.

\begin{pro}\label{o-operator-tensor-form}
Let $T:\g^*\lon\g$ be a linear map.
 \begin{itemize}
   \item[\rm(i)] $T$ is an $\huaO$-operator on the bimodule  $(\g^*;\ad^*,-R^*)$  over a pre-Lie algebra $\g$ if and only if the tensor form $\bar{T}=\sum_{i=1}^{n}e_i\otimes T(e_i^*)\in\g\otimes\g$ satisfies
\begin{eqnarray*}
\nonumber\llbracket \bar{T},\bar{T}\rrbracket = 0.
\end{eqnarray*}
 \item[\rm(ii)] $T=T^*$ if and only if $\bar{T}\in\Sym^2(\g)$.
 \end{itemize}
\end{pro}

\begin{proof}
Since $\Psi$ is a gLa isomorphism from $\oplus_{k\ge1}(\wedge^{k-1}\g\otimes\g \otimes\g)$ to $(C^*(\g^*,\g),[-,-]_\mu)$, we deduce that $[T,T]_\mu=0$ if and only if $\llbracket \bar{T},\bar{T}\rrbracket = 0.$ The rest conclusion is obvious.
\end{proof}

\begin{defi}
Let $(\g,\cdot_\g)$ be a pre-Lie algebra and $r\in\Sym^2(\g)$. The equation
 \begin{equation}\label{S-equation1}
\llbracket r,r\rrbracket=0
\end{equation}
 is call {\bf $S$-equation} in $\g$ and $r$ is called an {\bf $\frks$-matrix}.
\end{defi}
\begin{rmk}
The definition of $S$-equation and $\frks$-matrix had been introduced by Bai in \cite{Left-symmetric bialgebras} as an analogue of the classical Yang-Baxter equation and $r$-matrix respectively. But we obtain these notions through the gLa on the tensor space $\oplus_{k\ge1}(\wedge^{k-1}\g\otimes\g \otimes\g)$, which is different from the approach given by Bai.
\end{rmk}
 \emptycomment{For any $r\in\g\otimes \g$, the linear map $r^\sharp:\g^*\longrightarrow \g$ is given by
$$\langle r^\sharp(\xi),\eta\rangle=r(\xi,\eta),\quad \forall  ~\xi,\eta\in \g^*.$$
Let $r\in\Sym^2(\g)$, then
\begin{eqnarray}
\nonumber\langle \half[\hat{r}^\sharp,\hat{r}^\sharp]_{\hat{\mu}}(\xi,\eta),\zeta\rangle&=&\langle r^\sharp(\xi)\cdot_\g r^\sharp(\eta)-r^\sharp(\ad^*_{r^\sharp(\xi)}\eta-R^*_{r^\sharp(\eta)}\xi) ,\zeta\rangle\\
\label{eq:twist-s-matrix}&=&-\langle\xi,r^\sharp(\zeta)\cdot_\g r^\sharp(\eta)\rangle+\langle\zeta,r^\sharp(\xi)\cdot_\g r^\sharp(\eta)\rangle+\langle\eta,[r^\sharp(\xi),r^\sharp(\zeta)]_\g\rangle\\
\nonumber&=&\llbracket r,r\rrbracket(\xi,\zeta,\eta).
\end{eqnarray}}

\emptycomment{There is a close relationship between $\frks$-matrices and $\GRB$-operators.
\begin{pro}\label{pro:assrmatrix-GRB1}
 Let $(\g,\cdot_\g)$ be a pre-Lie algebra.  Then $r\in\Sym^2(\g)$ is an $\frks$-matrix if and only if $[\hat{r}^\sharp,\hat{r}^\sharp]_{\hat{\mu}}=0$, or equivalently, $r^\sharp$ is an $\GRB$-operator on the bimodule $(\g^*;\ad^*,-R^*)$ over $\g$, where $\mu$ is the semi-direct pre-Lie algebra structure of $\g\ltimes_{\ad^*,-R^*} \g^*$.
\end{pro}}

For any $r\in\g\otimes \g$, define a linear map $r^\sharp:\g^*\longrightarrow \g$ by
$$\langle r^\sharp(\xi),\eta\rangle=r(\xi,\eta),\quad \forall  ~\xi,\eta\in \g^*.$$
Assume that $r\in\Sym^2(\g)$ is an $\frks$-matrix. By Proposition \ref{o-operator-tensor-form}, $r^\sharp:\g^*\longrightarrow \g$ is an $\huaO$-operator on the bimodule $(\g^*;\ad^*,-R^*)$. By Proposition \ref{pro:algrep}, we have
\begin{cor}{\rm(\cite{Left-symmetric bialgebras})}
  Let $(\g,\cdot_\g)$ be a pre-Lie algebra and $r\in\Sym^2(\g)$ an $\frks$-matrix. Then $(\g^*,\cdot^{r^\sharp})$ is a pre-Lie algebra, where $\cdot^{r^\sharp}$ is given by
\begin{equation}\label{eq:s-pre-Lie operation}
  \xi\cdot^{r^\sharp} \eta= \ad^*_{r^\sharp(\xi)}\eta-R^*_{r^\sharp(\eta)}\xi,\quad\forall~\xi,\eta\in \g^*,
\end{equation}
and $r^\sharp$ is a homomorphism from the pre-Lie algebra $(\g^*,\cdot^{r^\sharp})$ to the pre-Lie algebra $(\g,\cdot_\g)$. We denote this pre-Lie algebra by $\g^*_{r^\sharp}$.
\end{cor}

By Theorem \ref{thm:isomorphism-invariant-twilled}, we have
\begin{cor}{\rm(\cite{Left-symmetric bialgebras})}
  Let $(\g,\cdot_\g)$ be a pre-Lie algebra and $r\in\Sym^2(\g)$ an $\frks$-matrix. Then $((\frkd=\g\oplus\g^*_{r^\sharp},\cdot_\frkd,(-,-)_-),\g,\g^*_{r^\sharp})$ is a Manin triple for the pre-Lie algebra $\g$ and $\g^*_{r^\sharp}$, where the  pre-Lie algebra multiplication $\cdot_\frkd$ is given by
  \begin{equation}
     (x,\xi)\cdot_\frkd(y,\eta)=(x\cdot_{\g}y+r^\sharp(\xi)\cdot_\g y+r^\sharp(R^*_y\xi)+x\cdot_\g r^\sharp(\eta)-r^\sharp(\ad^*_x \eta),\ad^*_x \eta-R^*_y\xi+\xi\cdot^{r^\sharp} \eta),
   \end{equation}
  and the bilinear form $(-,-)_-$ is given by \eqref{symplectic bracket}.
\end{cor}

By Corollary \ref{cor:bialg}, we have
\begin{cor}{\rm(\cite{Left-symmetric bialgebras})}
 Let $(\g,\cdot_\g)$ be a pre-Lie algebra and $r\in\Sym^2(\g)$ an $\frks$-matrix. Then $(\g,\g^*_{r^\sharp})$ is a pre-Lie bialgebra.
\end{cor}

\section{Quasi-pre-Lie bialgebras and twisted $\frks$-matrices}\label{sec:Q}
In this section, we use the approach of twisting operations on pre-Lie algebras to study the quasi-pre-Lie bialgebras and quasi-Manin triples for pre-Lie algebras.
\begin{defi}\label{defi:quasi-pre-Lie bialgebra}
 A {\bf quasi-pre-Lie bialgebra} on $(\g,\g^*)$ consists of the following data:
 \begin{itemize}
\item[$\bullet$]A pre-Lie algebra $(\g^*,\cdot_{\g^*})$ and a linear map $\alpha:\g\rightarrow \g\otimes \g$ such that
$\langle\alpha(x),\xi\otimes \eta\rangle=\langle x, \xi\cdot_{\g^*} \eta\rangle$;
\item[$\bullet$]An operation $\cdot_{\g}:\g\otimes\g\rightarrow \g$ and a linear map $\beta:\g^*\rightarrow \g^*\otimes \g^*$ such that
$\langle\beta(\xi),x\otimes y\rangle=\langle\xi, x\cdot_\g y\rangle$;
\item[$\bullet$] An element $\Phi\in\wedge^2\g^*\otimes\g^*$
\end{itemize}
satisfying the following properties:
 \begin{itemize}
\item[${\rm(i)}$]for any $\xi\in\g^*$,
\begin{eqnarray*}
\nonumber&&(\beta\otimes 1)\beta(\xi)-(1\otimes\beta)\beta(\xi)-\tau\circ(\beta\otimes 1)\beta(\xi)+\tau\circ(1\otimes\beta)\beta(\xi)\\
&&= (L_\xi\otimes1\otimes1+1\otimes L_\xi\otimes1+1\otimes1\otimes \ad_\xi)\Phi,
\end{eqnarray*}
where $\tau:\otimes^3\g^*\rightarrow\otimes^3\g^* $ is defined by $\tau(\xi\otimes\eta\otimes\zeta)=\eta\otimes\xi\otimes\zeta$;
\item[$\rm(ii)$]$\dt(\Phi)=0$, where $\dt$ is the formal coboundary operator associated to the trivial bimodule on the structure $(\g,\cdot_\g)$;
\item[$\rm(iii)$]$\Phi(x,y,z)=\Phi(x,z,y)-\Phi(y,z,x)$;
\item[$\rm(iv)$]$\alpha([x,y]_\g)=(L_x\otimes1+1\otimes\ad_x)\alpha(y)-(L_y\otimes1+1\otimes\ad_y)\alpha(x)$, where $[x,y]_\g=x\cdot_\g y-y\cdot_\g x$ for $x,y\in\g$;
\item[$\rm(v)$]$\beta$ is a $1$-cocycle associated to the module $(\g^*\otimes\g^*;L\otimes1+1\otimes\ad)$ over the sub-Lie algebra $({\g^*},[-,-]_{\g^*})$.
\end{itemize}
We denote a quasi-pre-Lie bialgebra by $(\g,\g^*,\Phi)$.
\end{defi}

\begin{defi}
A   {\bf quasi-Manin triple for pre-Lie algebras}
 is a triple $((\frkd,\cdot_\frkd, (-,-)_-),\g_1,\g_2)$, where $(\frkd,\cdot_\frkd,(-,-)_-)$ is an even dimensional quadratic pre-Lie algebra, $\g_1$ is a subspace and $\g_2$ is a pre-Lie subalgebra, both isotropic with respect to $(-,-)_-$, and $\frkd=\g_1\oplus\g_2$ as vector spaces.
\end{defi}

\begin{rmk}
Let $((\frkd,\cdot_\frkd, (-,-)_-),\g_1,\g_2)$ be a quasi-Manin triple. Then $\frkd=\g_1\oplus\g_2$ is a quasi-twilled pre-Lie algebra. Thus a  quasi-Manin triple can be seen as a quasi-twilled pre-Lie algebra with a nondegenerate invariant skew-symmetric bilinear form.
\end{rmk}
\begin{ex}\label{ex:quasi-quadratic pre-Lie}{\rm
 Let $(\g,\cdot_\g)$ be a pre-Lie algebra. If $\phi:\g\times\g \rightarrow \g^*$ is a $2$-cocycle associated to the bimodule $(\g^*;\ad^*,-R^*)$ and satisfies
 \begin{equation}\label{eq:phi-invariant condition}
   \langle\phi(x,y),z\rangle-\langle \phi(x,z),y\rangle+\langle \phi(z,x),y\rangle=0.
 \end{equation}
 Then $(\frkd=\g\ltimes_{\ad^*,-R^*,\phi} \g^*,(-,-)_-)$ is a quadratic pre-Lie algebra, where the nondegenerate invariant skew-symmetric bilinear form $(-,-)_-$ is given by \eqref{symplectic bracket}.}
\end{ex}

\begin{thm}\label{thm:equivalence2}
  There is a one-to-one correspondence between  quasi-Manin triples for pre-Lie algebras and quasi-pre-Lie bialgebras.
\end{thm}
\begin{proof}
Let $(\g,\g^*,\Phi)$ be a quasi-pre-Lie bialgebra. We show that $((\frkd=\g\oplus\g^*,\cdot_\frkd,(-,-)_-),\g,\g^*)$ is a quasi-Manin triple, where the quasi-twilled pre-Lie algebra structure  on $\g\oplus \g^*$ is given by
       \begin{eqnarray}\label{eq:bracket}
    (x,\xi)\cdot_\frkd (y,\eta)=(x\cdot_\g y+\ad^*_\xi y-R^*_\eta x,\ad_x^*\eta-R_y^*\xi+\xi\cdot_{\g^*}\eta+\phi(x,y)),
    \end{eqnarray}
    where $\phi(x,y):=\Phi(x,-,y)$ and $(-,-)_-$ is given by \eqref{symplectic bracket}.

By \eqref{eq:bracket} and condition (iii) in the definition of quasi-pre-Lie bialgebra, the bilinear form $(-,-)_-$ given by \eqref{symplectic bracket} is invariant.

Furthermore, $(\frkd=\g\oplus\g^*,\cdot_\frkd)$ is a quasi-twilled pre-Lie algebra structure, which is obtained by the following relations:
\begin{eqnarray*}
 \mbox{Conditions (i) and (ii) }&\Longrightarrow& ((x,0),(y,0),(z,0))= ((y,0),(x,0),(z,0));\\
  \mbox{Conditions (i) and (iv) }&\Longrightarrow&((x,0),(y,0),(0,\zeta))= ((y,0),(x,0),(0,\zeta));\\
  \mbox{Conditions (i) and (v)}&\Longrightarrow&((0,\xi),(y,0),(z,0))= ((y,0),(0,\xi),(z,0));\\
  \mbox{$(\g^*,\cdot_{\g^*})$ is a pre-Lie algebra}&\Longrightarrow& ((0,\xi),(0,\eta),(0,\zeta))= ((0,\eta),(0,\xi),(0,\zeta));\\
  \mbox{$(\g^*,\cdot_{\g^*})$ is a pre-Lie algebra and (v) }&\Longrightarrow&((0,\xi),(0,\eta),(0,z))= ((0,\eta),(0,\xi),(0,z));\\
  \mbox{$(\g^*,\cdot_{\g^*})$ is a pre-Lie algebra and (iv) }&\Longrightarrow&((x,0),(0,\eta),(0,\zeta))= ((0,\eta),(0,x),(0,\zeta)).
\end{eqnarray*}

Conversely, let $((\frkd,\cdot_\frkd, (-,-)_-),\g_1,\g_2)$ be a quasi-Manin triple. Since the pairing $(-,-)_-$ is nondegenerate, $\g_2$ is isomorphic to $\g_1^*$, via $ \langle \xi,y\rangle=(\xi,y)_-$ for all $y\in\g_1,~\xi\in\g_2$. Under this isomorphism, the skew-symmetric bilinear form $(\cdot,\cdot)_-$ on $\frkd$ is given by \eqref{symplectic bracket}.

By the invariance of the bilinear form $(-,-)_-$, we deduce that the multiplication between $\g_1$ and $\g_2$ is given by \eqref{eq:bracket}.

The rest is a similar proof of the above relations. We omit the details.
\end{proof}

\begin{cor}
 Let $(\g,\g^*,\Phi)$ be a quasi-pre-Lie bialgebra. Then $({\g}^c\oplus {\g^*}^c,[-,-])$ is a Lie algebra, where the bracket is defined by
    \begin{eqnarray*}
    [x+\xi,y+\eta]=[\xi,\eta]_{\g^*}+L^*_x\xi-L^*_y\eta+L^*_\xi y-L^*_\eta x+[x,y]_\g+\phi(x,y)-\phi(y,x),
    \end{eqnarray*}
    where $\phi(x,y):=\Phi(x,-,y)$ and $x,y\in\g,\xi,\eta\in \g^*$. Furthermore, $({\g}^c\oplus {\g^*}^c,\omega)$ is a symplectic Lie algebra, where $\omega$ is given by
    \begin{equation*}\label{eq:defiomega}
     \omega(x+\xi,y+\eta)=\langle\xi,y\rangle-\langle\eta,x\rangle,\quad \forall x,y\in\g,~\xi,\eta\in\g^*.
   \end{equation*}
\end{cor}
\begin{proof}
  By Theorem \ref{thm:equivalence2}, the double structure given by \eqref{eq:bracket} of the quasi-pre-Lie bialgebra $(\g,\g^*,\Phi)$ is a quasi-twilled pre-Lie algebra. The first claim follows that $({\g}^c\oplus {\g^*}^c,[-,-])$ is the sub-adjacent Lie algebra of the above quasi-twilled pre-Lie algebra. The second claim follows the invariance of the non-degenerate bilinear form $(-,-)$.
\end{proof}

Let $T$ be a $\phi$-twisted $\huaO$-operator on a bimodule $(V;\huaL,\huaR)$ over a pre-Lie algebra $(\g,\cdot_\g)$. Let $\pi^\phi$ be the structure of the quasi-pre-Lie algebra $\g\ltimes_{\ad^*,-R^*,\phi}\g^*$. By Corollary \ref{cor:twist-o-quasi-pre-Lie},   $(( \g\oplus\g^*,\pi^{T,\phi}),\g,\g^*)$  is a quasi-twilled pre-Lie algebra. Moreover, by Corollary \ref{twisting-isomorphism},
$e^{\hat{T}}:( \g\oplus\g^*,\pi^T)\lon( \g\oplus\g^*,\pi)$ is an isomorphism between  pre-Lie algebras.

By Corollary \ref{cor:twist-pre-Lie operation} and Corollary \ref{cor:twist-o-quasi-pre-Lie}, we have
\begin{pro}\label{pro:twist-O-quasi-pre-Liebia}
Let $T:\g^*\lon\g$ be a $\phi$-twisted $\GRB$-operator on the bimodule $(\g^*;\ad^*,-R^*)$ over $\g$. Then $\g^*_{T,\phi}=(\g^*,\cdot^{T,\phi})$ is a pre-Lie algebra, where $\cdot^{{T,\phi}}$ is given by
\begin{equation}\label{eq:twist-s-pre-Lie operation}
  \xi\cdot^{T,\phi} \eta= \ad^*_{T(\xi)}\eta-R^*_{T(\eta)}\xi+\phi(T(\xi),T(\eta)),\quad\forall~\xi,\eta\in \g^*.
\end{equation}
Moreover, the twilled pre-Lie algebra $\pi^{T,\phi}$ on $\g\oplus\g^*$ is given by
  \begin{eqnarray*}
     (x,\xi)\ast(y,\eta)&=&(x\cdot_{\g}y+{\ad_\xi^*}^T y-{R_\eta^*}^T x-T(\phi(x,y))-T(\phi(T(\xi),y))\\
     &&-T(\phi(x,T(\eta))),\ad^*_x \eta-R^*_y\xi+\xi\cdot^{T,\phi} \eta+\phi(x,y)+\phi(T(\xi),y)+\phi(x,T(\eta)),
   \end{eqnarray*}
   where ${\ad^*}^T$ and $-{R^*}^T$ are given by \eqref{eq:dual bimodule in bialgebra} and $x,y\in\g,\xi,\eta\in \g^*$.
\end{pro}

Similar to the Lemma \ref{lem:O-symmetric}, we have
\begin{lem}
 Let $T:\g^*\lon\g$ be a $\phi$-twisted $\GRB$-operator on the bimodule $(\g^*;\ad^*,-R^*)$ over $\g$. Then $e^{\hat{T}}$ preserves the bilinear form $(-,-)_-$ given by \eqref{symplectic bracket} if and only if $T =T^*$.
\end{lem}
Similar to the Theorem \ref{thm:isomorphism-invariant-twilled}, we have
\begin{thm}\label{thm:isomorphism-invariant-quais-twilled}
 Let $T:\g^*\lon\g$ be a $\phi$-twisted $\GRB$-operator on the bimodule $(\g^*;\ad^*,-R^*)$ over $\g$, $\phi$ satisfies \eqref{eq:phi-invariant condition} and $T=T^*$.
Then $(\frkd=\g\oplus\g^*,\pi^{T,\phi})$  is a quadratic pre-Lie algebra with the invariant bilinear form $(-,-)_-$ given by \eqref{symplectic bracket} and $e^{\hat{T}}$ is an isomorphism from the quadratic pre-Lie algebra $(\frkd=\g\oplus\g^*,\pi^{T,\phi})$ to $(\frkd=\g\oplus\g^*,\pi^\phi)$.
 \end{thm}

By Theorem \ref{thm:equivalence2}, Proposition \ref{pro:twist-O-quasi-pre-Liebia} and Theorem \ref{thm:isomorphism-invariant-quais-twilled}, we obtain
\begin{cor}\label{cor:quasi-pre-Liebialg}
   Let $T:\g^*\lon\g$ be a $\phi$-twisted $\GRB$-operator on the bimodule $(\g^*;\ad^*,-R^*)$ over $\g$, $\phi$ satisfies \eqref{eq:phi-invariant condition} and $T=T^*$. Then $(\g,\g^*_{T,\phi},\Phi)$ is a quasi-pre-Lie bialgebra, where $\Phi(x,y,z):=\langle\phi(x,z),y\rangle$ for $x,y,z\in\g$.
\end{cor}

\begin{lem}\label{lem:twist-prepare}
  Let $(\g,\cdot_\g)$ be a pre-Lie algebra. Then $\phi:\g\times\g \rightarrow \g^*$ is a $2$-cocycle associated to the bimodule $(\g^*;\ad^*,-R^*)$ and satisfies \eqref{eq:phi-invariant condition} if and only if $\Phi\in \wedge^2\g^*\otimes\g^*$ satisfies $\dt(\Phi)=0$ and
 \begin{equation}\label{eq:quasi-preLie bialgebra extra}
\Phi(x,y,z)=\Phi(x,z,y)-\Phi(y,z,x),
  \end{equation}
  where $\Phi(x,y,z):=\langle\phi(x,z),y\rangle$ for $x,y,z\in\g$.
\end{lem}
\begin{proof}
  It follows that $\phi$ is a $2$-cocycle associated to the bimodule $(\g^*;\ad^*,-R^*)$ if and only if $\Phi\in \wedge^2\g^*\otimes\g^*$ satisfies $\dt(\Phi)=0$ and $\phi$ satisfies \eqref{eq:phi-invariant condition} if and only if $\Phi$ satisfies \eqref{eq:quasi-preLie bialgebra extra}.
\end{proof}

\begin{pro}\label{pro:twisted s-matrix equivalent}
 Let $(\g,\cdot_\g)$ be a pre-Lie algebra and $r\in\Sym^2(\g)$. Then $r^\sharp:\g^*\rightarrow \g$ is a $\phi$-twisted $\GRB$-operator on the bimodule $(\g^*;\ad^*,-R^*)$ over $\g$ and $\phi$ satisfies \eqref{eq:phi-invariant condition} if and only if $r$ satisfies
 \begin{equation}\label{eq:twist s-matrix}
 \half\llbracket r,r\rrbracket=(\otimes^3r^\sharp)\Phi
 \end{equation}
 and $\Phi\in \wedge^2\g^*\otimes\g^*$ satisfies $\dt(\Phi)=0$ and \eqref{eq:quasi-preLie bialgebra extra}, where $\Phi(x,y,z):=\langle\phi(x,z),y\rangle$ for any $x,y,z\in\g$.
\end{pro}
\begin{proof}
We only show that $r^\sharp:\g^*\rightarrow \g$ is a $\phi$-twisted $\GRB$-operator on the bimodule $(\g^*;\ad^*,-R^*)$ if and only if \eqref{eq:twist s-matrix} holds. The rest follows by Lemma \ref{lem:twist-prepare}.

By direct calculations, we have
  \begin{eqnarray*}
\half\llbracket r,r\rrbracket(\xi,\zeta,\eta)=\langle r^\sharp(\xi)\cdot_\g r^\sharp(\eta)-r^\sharp(\ad^*_{r^\sharp(\xi)}\eta-R^*_{r^\sharp(\eta)}\xi) ,\zeta\rangle.
  \end{eqnarray*}
  On the other hand,
  \begin{eqnarray*}
  (\otimes^3r^\sharp)\Phi(\xi,\zeta,\eta)&=&\Phi(r^\sharp(\xi),r^\sharp(\zeta),r^\sharp(\eta))=\langle \phi(r^\sharp(\xi),r^\sharp(\eta)),r^\sharp(\zeta)\rangle\\
  &=&\langle r^\sharp(\phi(r^\sharp(\xi),r^\sharp(\eta))),\zeta\rangle.
  \end{eqnarray*}
  Thus $\half\llbracket r,r\rrbracket(\xi,\zeta,\eta)= (\otimes^3r^\sharp)\Phi(\xi,\zeta,\eta)$ holds
   if and only if
 $$r^\sharp(\xi)\cdot_\g r^\sharp(\eta)-r^\sharp(\ad^*_{r^\sharp(\xi)}\eta-R^*_{r^\sharp(\eta)}\xi)=r^\sharp(\phi(r^\sharp(\xi),r^\sharp(\eta))).$$
 The claim follows.
\end{proof}

\begin{defi}
Let $(\g,\cdot_\g)$ be a pre-Lie algebra and $\Phi\in \wedge^2\g^*\otimes\g^*$ is $\dt$-closed and satisfies \eqref{eq:quasi-preLie bialgebra extra}. An element $r\in\Sym^2(\g)$ is called a {\bf $\Phi$-twisted $\frks$-matrix}, or simply twisted $\frks$-matrix if $r$ satisfies \eqref{eq:twist s-matrix}.
\end{defi}

By Proposition \ref{pro:twisted s-matrix equivalent} and Corollary \ref{cor:twist-pre-Lie operation}, we have
\begin{cor}
  Let $(\g,\cdot_\g)$ be a pre-Lie algebra and $r\in\Sym^2(\g)$ a $\Phi$-twisted $\frks$-matrix. Set $\phi(x,y)=\Phi(x,-,y)$. Then $(\g^*,\cdot^{r^\sharp,\phi})$ is a pre-Lie algebra, where $\cdot^{{r^\sharp,\phi}}$ is given by
\begin{equation}\label{eq:twist-s-pre-Lie operation}
  \xi\cdot^{r^\sharp,\phi} \eta= \ad^*_{r^\sharp(\xi)}\eta-R^*_{r^\sharp(\eta)}\xi+\phi(r^\sharp(\xi),r^\sharp(\eta)),\quad\forall~\xi,\eta\in \g^*
\end{equation}
and $r^\sharp$ is a homomorphism from the pre-Lie algebra $(\g^*,\cdot^{r^\sharp,\phi})$ to the pre-Lie algebra $(\g,\cdot_\g)$. We denote this pre-Lie algebra by $\g^*_{r^\sharp,\phi}$.
\end{cor}
By Theorem \ref{thm:isomorphism-invariant-quais-twilled}, we have
\begin{cor}
  Let $(\g,\cdot_\g)$ be a pre-Lie algebra and $r\in\Sym^2(\g)$ a $\Phi$-twisted $\frks$-matrix. Then $((\frkd=\g\oplus\g^*_{r^\sharp,\phi},\cdot_\frkd,(-,-)_-),\g,\g^*_{r^\sharp})$ is a quasi-Manin triple, where
  the pre-Lie algebra multiplication $\cdot_\frkd$ is given by
  \begin{eqnarray*}
    {(x,\xi)\cdot_\frkd(y,\eta)}&=&(x\cdot_{\g}y+r^\sharp(\xi)\cdot_\g y+x\cdot_\g r^\sharp(\eta)
     +r^\sharp(R^*_y\xi)-r^\sharp(\ad^*_x \eta)\\
     &&-r^\sharp(\phi(r^\sharp(\xi),y))-r^\sharp(\phi(x, r^\sharp(\eta)))-r^\sharp(\phi(x,y)),\ad^*_x \eta-R^*_y\xi\\
     &&+\phi(r^\sharp(\xi),y)+\phi(x,r^\sharp(\eta))+\xi\cdot^{r^\sharp,\phi} \eta+\phi(x,y)),
   \end{eqnarray*}
   where $x,y\in\g,\xi,\eta\in \g^*$, and the bilinear form $(-,-)_-$ is given by \eqref{symplectic bracket}.
\end{cor}

By Corollary \ref{cor:quasi-pre-Liebialg}, we have
\begin{cor}
 Let $(\g,\cdot_\g)$ be a pre-Lie algebra and $r\in\Sym^2(\g)$ a $\Phi$-twisted $\frks$-matrix. Then $(\g,\g^*_{r^\sharp,\phi},\Phi)$ is a quasi-pre-Lie bialgebra.
\end{cor}

In the following, we use symplectic Lie algebras to construct quasi-pre-Lie bialgebras, which is parallel to that a Cantan $3$-form on a semisimple Lie algebra gives a quasi-Lie bialgebra. First, we recall a useful lemma.

\emptycomment{\begin{defi}
  A Lie algebra $(\g,[-,-]_\g)$ is called a {\bf symplectic Lie algebra} if there is a nondegenerate skew-symmetric $2$-cocycle $\omega$ on $\g$, i.e.
  $$\omega([x,y]_\g,z)+\omega([y,z]_\g,x)+\omega([z,x]_\g,y)=0,\quad\forall~x,y,z\in\g.$$
\end{defi}}

\begin{lem}{\rm(\cite{symplectic Lie algebras})}
  Let $(\g,[-,-]_\g,\omega)$ be a symplectic Lie algebra. Then there exists a compatible pre-Lie algebra structure $\cdot_\g$ on $\g$ given by
  \begin{equation}\label{eq:pre-Lie-symplectic}
    \omega(x\cdot_\g y,z)=-\omega(y,[x,z]_\g),\quad\forall~x,y,z\in\g.
  \end{equation}
\end{lem}

\begin{pro}\label{pro:symplectic2}
  Let $(\g,[-,-]_\g,\omega)$ be a symplectic Lie algebra and $(\g,\cdot_\g)$ its compatible pre-Lie algebra. Define $\Phi\in\wedge^2\g^*\otimes\g^*$ by
  \begin{equation}\label{eq:3-form-symplectic}
    \Phi(x,y,z)=\omega([x,y]_\g,z),\quad\forall~x,y,z\in\g.
  \end{equation}
  Then $\Phi$ is $\dt$-closed and satisfies \eqref{eq:quasi-preLie bialgebra extra}.
\end{pro}
\begin{proof}
  By \eqref{eq:pre-Lie-symplectic}, we can show that
  \begin{eqnarray*}
    \dt(\Phi)(x,y,z,w)&=&2\omega([x,[y,z]_\g]_\g+[z,[x,y]_\g]_\g+[y,[z,x]_\g]_\g,w);\\
    \Phi(x,y,z)-\Phi(x,z,y)+\Phi(y,z,x)&=&\omega([x,y]_\g,z)+\omega([y,z]_\g,x)+\omega([z,x]_\g,y).
  \end{eqnarray*}
  By the Jacobi identity of the Lie algebra $\g$, $\dt(\Phi)=0$. Since $\omega$ is a symplectic structure on the Lie algebra $\g$, $\Phi(x,y,z)=\Phi(x,z,y)-\Phi(y,z,x).$
\end{proof}
By Example \ref{ex:quasi-quadratic pre-Lie}, Lemma \ref{lem:twist-prepare} and Proposition \ref{pro:symplectic2}, we have
\begin{pro}
   Let $(\g,[-,-]_\g,\omega)$ be a symplectic Lie algebra and $(\g,\cdot_\g)$ its compatible pre-Lie algebra. Then $(\frkd=\g\oplus\g^*,\cdot_\frkd,(-,-)_-)$ is a quadratic pre-Lie algebra, where the multiplication $\cdot_\frkd$ is given by
   \begin{equation}\label{eq:standard quadratic-pre-Lie}
      (x,\xi)\cdot_\frkd(y,\eta)=(x\cdot_\g y,\ad^*_x\eta-R^*_y \xi+\phi(x,y)),\quad\forall~x,y\in\g,\xi,\eta\in\g^*,
   \end{equation}
   where $\phi(x,y)=\Phi(x,-,y)$ and $\Phi$ is given by \eqref{eq:3-form-symplectic}. Furthermore, $(\g,\g^*,\Phi)$ is a quasi-pre-Lie bialgebra.
\end{pro}
\begin{ex}
Let $\g$ be the 4-dimensional Lie algebra with basis $\{e_1, e_2, e_3, e_4\}$ and the bracket
$$[e_1,e_2]=e_3.$$
Let $\{e_1^*,e_2^*,e_3^*,e_4^*\}$ be the dual basis and $\omega$ a nondegenerate skew-symmetric bilinear form given by
  $$\omega=e^*_1\wedge e^*_4+e^*_2\wedge e^*_3.$$
  Then $\omega$ is a symplectic form on the Lie algebra $(\g,[-,-])$. It is straightforward to verify that the compatible pre-Lie algebra structure on $\g$ is given by
  $$e_1\cdot e_2=e_3,\quad e_2\cdot e_2=e_4.$$

  By straightforward calculations, the $\Phi$ in \eqref{eq:3-form-symplectic} is given by
  $$\Phi=-e^*_1\wedge e^*_2\otimes e^*_2.$$
  Furthermore, the pre-Lie algebra multiplication in \eqref{eq:standard quadratic-pre-Lie} is given by
  \begin{eqnarray*}
    (e_1,0)\ast (e_2,0)&=&(e_3,-e^*_2),\quad  (e_2,0)\ast(e_2,0)=(-e_4,e^*_1), \quad(e_1,0)\ast (0,e^*_3)=(0,-e^*_2),\\
(0, e^*_3)\ast (e_2, 0)&=&(0,e^*_1),\quad (0, e^*_4)\ast (e_2, 0)=(0,-e^*_2).
  \end{eqnarray*}
  Then $r$ given as follows
\begin{itemize}
\item[\rm(1)] $r=r_{11}e_1\odot e_2+r_{13}e_1\odot e_3+r_{14}e_1\odot e_4+r_{33}e_3\odot e_3+r_{34}e_3\odot e_4+r_{44}e_4\odot e_4$;
\item[\rm(2)] $r=r_{13}e_1\odot e_3+e_1\odot e_4-e_2\odot e_3+r_{33}e_3\odot e_3+r_{34}e_3\odot e_4+r_{44}e_4\odot e_4$;
\item[\rm(3)]$r=r_{13}e_1\odot e_3-\frac{r_{23}}{2+r_{23}}e_1\odot e_4+r_{23}e_2\odot e_3+r_{33}e_3\odot e_3+r_{34}e_3\odot e_4+r_{44}e_4\odot e_4$;
\item[\rm(4)]
   $r=-\frac{r^2_{23}}{r_{24}}e_1\odot e_3-r_{23}e_1\odot e_4+r_{23}e_2\odot e_3+r_{24}e_2\odot e_4+r_{33}e_3\odot e_3+r_{34}e_3\odot e_4+r_{44}e_4\odot e_4$
\end{itemize}
are $\Phi$-twisted $\frks$-matrices, where $r_{i,j}$ are constants and $e_i\odot e_j:=e_i\otimes e_j+e_j\otimes e_i$ for $1\leq i,j\leq 4$.
\end{ex}

\noindent
{\bf Acknowledgements.} This research is supported by NSFC (11901501). I give my warmest thanks to Chengming Bai, Yunhe Sheng and Rong Tang for very useful comments and discussions.

 \end{document}